\documentclass[11pt]{article}

\usepackage{amsmath}
\makeatletter
\let\save@mathaccent\mathaccent
\newcommand*\if@single[3]{%
  \setbox0\hbox{${\mathaccent"0362{#1}}^H$}%
  \setbox2\hbox{${\mathaccent"0362{\kern0pt#1}}^H$}%
  \ifdim\ht0=\ht2 #3\else #2\fi
  }
\newcommand*\rel@kern[1]{\kern#1\dimexpr\macc@kerna}
\newcommand*\widebar[1]{\@ifnextchar^{{\wide@bar{#1}{0}}}{\wide@bar{#1}{1}}}
\newcommand*\wide@bar[2]{\if@single{#1}{\wide@bar@{#1}{#2}{1}}{\wide@bar@{#1}{#2}{2}}}
\newcommand*\wide@bar@[3]{%
  \begingroup
  \def\mathaccent##1##2{%
    \let\mathaccent\save@mathaccent
    \if#32 \let\macc@nucleus\first@char \fi
    \setbox\z@\hbox{$\macc@style{\macc@nucleus}_{}$}%
    \setbox\tw@\hbox{$\macc@style{\macc@nucleus}{}_{}$}%
    \dimen@\wd\tw@
    \advance\dimen@-\wd\z@
    \divide\dimen@ 3
    \@tempdima\wd\tw@
    \advance\@tempdima-\scriptspace
    \divide\@tempdima 10
    \advance\dimen@-\@tempdima
    \ifdim\dimen@>\z@ \dimen@0pt\fi
    \rel@kern{0.6}\kern-\dimen@
    \if#31
      \overline{\rel@kern{-0.6}\kern\dimen@\macc@nucleus\rel@kern{0.4}\kern\dimen@}%
      \advance\dimen@0.4\dimexpr\macc@kerna
      \let\final@kern#2%
      \ifdim\dimen@<\z@ \let\final@kern1\fi
      \if\final@kern1 \kern-\dimen@\fi
    \else
      \overline{\rel@kern{-0.6}\kern\dimen@#1}%
    \fi
  }%
  \macc@depth\@ne
  \let\math@bgroup\@empty \let\math@egroup\macc@set@skewchar
  \mathsurround\z@ \frozen@everymath{\mathgroup\macc@group\relax}%
  \macc@set@skewchar\relax
  \let\mathaccentV\macc@nested@a
  \if#31
    \macc@nested@a\relax111{#1}%
  \else
    \def\gobble@till@marker##1\endmarker{}%
    \futurelet\first@char\gobble@till@marker#1\endmarker
    \ifcat\noexpand\first@char A\else
      \def\first@char{}%
    \fi
    \macc@nested@a\relax111{\first@char}%
  \fi
  \endgroup
}
\makeatother
\usepackage{yhmath}
\usepackage{amssymb}
\usepackage{amsfonts}
\usepackage{mathrsfs}
\usepackage{graphicx}
\usepackage{color,xcolor}
\usepackage{float}
\usepackage{subcaption}
\usepackage[normalem]{ulem}
\usepackage[linktocpage,colorlinks,linkcolor=blue,anchorcolor=blue, citecolor=blue,urlcolor=blue]{hyperref}
\usepackage{pdfpages}
\usepackage[inline]{enumitem}
\usepackage{multirow}
\usepackage{makecell}
\usepackage{breakcites}
\usepackage{bbm}
\usepackage{wrapfig}
\usepackage{wrapstuff}
\usepackage[thinc]{esdiff}
\allowdisplaybreaks[4]
\usepackage{rotating}
\usepackage{scalerel,stackengine}
\usepackage{cases}
\usepackage{mleftright}
\usepackage[inline]{asymptote}
\usepackage{lmodern}
\usepackage{microtype}
\usepackage{cite}
\newtheorem{theorem}{Theorem}[section]
\newtheorem{definition}[theorem]{Definition}
\newtheorem{lemma}[theorem]{Lemma}
\newtheorem{example}[theorem]{Example}
\newtheorem{corollary}[theorem]{Corollary}
\newtheorem{conjecture}[theorem]{Conjecture}
\newtheorem{proposition}[theorem]{Proposition}

\newtheorem{algorithm}[theorem]{Algorithm}

\def\bd{\boldsymbol}
\def\pn{\par\smallskip\noindent}

\newenvironment{myproof}[1] {\pn {\em Proof of {#1}.}}{\hfill $\Box$ \vskip 0.2truein}

\DeclareMathAlphabet\mathbfcal{OMS}{cmsy}{b}{n}

\newcommand{\R}{\mathbb{R}}

\newcommand{\bbP}{\mathbb{P}}

\newcommand{\F}{\mathcal{F}}

\newcommand{\bbB}{\mathbb{B}}
\newcommand{\bbF}{\mathbb{F}}

\newcommand{\bbD}{\mathbb{D}}

\newcommand{\bbS}{\mathbb{S}}

\newcommand{\bbX}{\mathbb{X}}
\newcommand{\bbY}{\mathbb{Y}}

\newcommand{\bA}{\boldsymbol{A}}

\newcommand{\bdd}{\boldsymbol{d}}

\newcommand{\bI}{\boldsymbol{I}}

\newcommand{\be}{\boldsymbol{e}}
\newcommand{\bg}{\boldsymbol{g}}

\newcommand{\bH}{\boldsymbol{H}}
\newcommand{\bO}{\boldsymbol{O}}
\newcommand{\bF}{\boldsymbol{F}}

\newcommand{\bx}{\boldsymbol{x}}
\newcommand{\by}{\boldsymbol{y}}
\newcommand{\bz}{\boldsymbol{z}}
\newcommand{\bu}{\boldsymbol{u}}
\newcommand{\bv}{\boldsymbol{v}}
\newcommand{\bw}{\boldsymbol{w}}

\newcommand{\bp}{\boldsymbol{p}}

\newcommand{\bq}{\boldsymbol{q}}
\newcommand{\bQ}{\boldsymbol{Q}}
\newcommand{\bs}{\boldsymbol{s}}

\newcommand{\bW}{\boldsymbol{W}}
\newcommand{\bX}{\boldsymbol{X}}
\newcommand{\bY}{\boldsymbol{Y}}
\newcommand{\bZ}{\boldsymbol{Z}}
\newcommand{\dd}{\textnormal{d}}

\newcommand{\dC}{\textnormal{C}}

\newcommand{\bbU}{\mathbb{U}}
\newcommand{\bbV}{\mathbb{V}}

\newcommand{\T}{\textnormal{T}}

\def\doublestroke#1{\pdfliteral{1 Tr .3 w}#1\pdfliteral{0 Tr 0 w}}

\usepackage{booktabs}
\usepackage{graphicx}
\usepackage{stmaryrd}
\usepackage{mathtools}
\def\multiset#1#2{\ensuremath{\left(\kern-.3em\left(\genfrac{}{}{0pt}{}{#1}{#2}\right)\kern-.3em\right)}}

\usepackage{multicol}
\usepackage{nicematrix}
\usepackage{stackengine}
\usepackage{rotating}

\stackMath

\usepackage{algorithm}
\usepackage[noend]{algorithmic}

\usepackage{stackengine}
\stackMath
\newcommand\tsup[2][2]{%
 \def\useanchorwidth{T}%
  \ifnum#1>1%
    \stackon[-.5pt]{\tsup[\numexpr#1-1\relax]{#2}}{\scriptscriptstyle\sim}%
  \else%
    \stackon[.5pt]{#2}{\scriptscriptstyle\sim}%
  \fi%
}

\newcommand\restr[2]{{
  \left.\kern-\nulldelimiterspace 
  #1 
  \littletaller 
  \right|_{#2} 
  }}

\newcommand{\littletaller}{\mathchoice{\vphantom{\big|}}{}{}{}}

\usepackage{tikz}
\usepackage{tikz-3dplot}
\usetikzlibrary{calc}
\usepackage{pgfplots}
\usepackage{xxcolor}
\pgfplotsset{compat=1.16}
\usetikzlibrary{external}
\usepgfplotslibrary{fillbetween}
\usetikzlibrary{patterns}
\pgfdeclareradialshading[tikz@ball]{ball}{\pgfqpoint{0bp}{0bp}}{%
 color(0bp)=(tikz@ball!0!white);
 color(7bp)=(tikz@ball!0!white);
 color(15bp)=(tikz@ball!70!black);
 color(20bp)=(black!70);
 color(30bp)=(black!70)}
\makeatother

\tikzset{viewport/.style 2 args={
    x={({cos(-#1)*1cm},{sin(-#1)*sin(#2)*1cm})},
    y={({-sin(-#1)*1cm},{cos(-#1)*sin(#2)*1cm})},
    z={(0,{cos(#2)*1cm})}
}}

\pgfplotsset{only foreground/.style={
    restrict expr to domain={rawx*\CameraX + rawy*\CameraY + rawz*\CameraZ}{-0.05:100},
}}
\pgfplotsset{only background/.style={
    restrict expr to domain={rawx*\CameraX + rawy*\CameraY + rawz*\CameraZ}{-100:0.05}
}}

\def\addFGBGplot[#1]#2;{
    \addplot3[#1,only background, opacity=0.25] #2;
    \addplot3[#1,only foreground] #2;
}

\usetikzlibrary{3d}
\usetikzlibrary{calc}
\usetikzlibrary{babel} 

\pgfmathsetmacro\xx{1/sqrt(2)}
\pgfmathsetmacro\xy{1/sqrt(6)}
\pgfmathsetmacro\zy{sqrt(2/3)}

\usetikzlibrary{3d}
\usetikzlibrary{calc, shadings} 
\usetikzlibrary{positioning,arrows.meta}
\usetikzlibrary{trees}
\usepgfplotslibrary{fillbetween}
\DeclareMathAlphabet\mathbfcal{OMS}{cmsy}{b}{n}

\makeatletter
\def\tikz@lib@cuboid@get#1{\pgfkeysvalueof{/tikz/cuboid/#1}}

\def\tikz@lib@cuboid@setup{%
   \pgfmathsetlengthmacro{\vxx}%
      {\tikz@lib@cuboid@get{xscale}*cos(\tikz@lib@cuboid@get{xangle})*1cm}
   \pgfmathsetlengthmacro{\vxy}%
      {\tikz@lib@cuboid@get{xscale}*sin(\tikz@lib@cuboid@get{xangle})*1cm}
   \pgfmathsetlengthmacro{\vyx}%
      {\tikz@lib@cuboid@get{yscale}*cos(\tikz@lib@cuboid@get{yangle})*1cm}
   \pgfmathsetlengthmacro{\vyy}%
      {\tikz@lib@cuboid@get{yscale}*sin(\tikz@lib@cuboid@get{yangle})*1cm}
   \pgfmathsetlengthmacro{\vzx}%
      {\tikz@lib@cuboid@get{zscale}*cos(\tikz@lib@cuboid@get{zangle})*1cm}
   \pgfmathsetlengthmacro{\vzy}%
      {\tikz@lib@cuboid@get{zscale}*sin(\tikz@lib@cuboid@get{zangle})*1cm}
}

\def\tikz@lib@cuboid@draw#1--#2--#3\pgf@stop{%
    \begin{scope}[join=bevel,x={(\vxx,\vxy)},y={(\vyx,\vyy)},z={(\vzx,\vzy)}]
       \begin{scope}[canvas is yz plane at x=#1]
          \draw[cuboid/all faces,cuboid/edges,cuboid/right face] 
                (0,0) -- ++(#2,0) -- ++(0,-#3) -- ++(-#2,0) -- cycle;
          \draw[cuboid/all grids,cuboid/right grid] (0,0) grid (#2,-#3);
       \end{scope}
       \begin{scope}[canvas is xy plane at z=0]
          \draw[cuboid/all faces,cuboid/edges,cuboid/front face] 
                (0,0) -- ++(#1,0) --  ++(0,#2) -- ++(-#1,0) -- cycle;
          \draw[cuboid/all grids,cuboid/front grid] (0,0) grid (#1,#2);
       \end{scope}
       \begin{scope}[canvas is xz plane at y=#2]
          \draw[cuboid/all faces,cuboid/edges,cuboid/top face] 
                (0,0) -- ++(#1,0) --  ++(0,-#3) -- ++(-#1,0) -- cycle;
          \draw[cuboid/all grids,cuboid/top grid] (0,0) grid (#1,-#3);
       \end{scope}
       \draw[cuboid/hidden edges] (0,#2,-#3) -- (0,0,-#3) -- (0,0,0) 
                (0,0,-#3) -- ++(#1,0,0);
       \begin{scope}[canvas is yz plane at x=#1]
          \draw[cuboid/all faces,cuboid/right face,cuboid/edges,fill opacity=0] 
                (0,0) -- ++(#2,0) -- ++(0,-#3) -- ++(-#2,0) -- cycle;
       \end{scope}
       \begin{scope}[canvas is xy plane at z=0]
          \draw[cuboid/all faces,cuboid/front face,cuboid/edges,fill opacity=0] 
                (0,0) -- ++(#1,0) --  ++(0,#2) -- ++(-#1,0) -- cycle;
       \end{scope}
       \begin{scope}[canvas is xz plane at y=#2]
          \draw[cuboid/all faces,cuboid/top face,cuboid/edges,fill opacity=0] 
                (0,0) -- ++(#1,0) --  ++(0,-#3) -- ++(-#1,0) -- cycle;
       \end{scope}
       \path (0,#2,0) coordinate (-left top front)
                      coordinate (-left front top)
                      coordinate (-top left front)
                      coordinate (-top front left)
                      coordinate (-front top left)
                      coordinate (-front left top);
       \path (0,#2,-#3) coordinate (-left top rear)
                        coordinate (-left rear top)
                        coordinate (-top left rear)
                        coordinate (-top rear left)
                        coordinate (-rear top left)
                        coordinate (-rear left top);
       \path (0,0,-#3) coordinate (-left bottom rear)
                       coordinate (-left rear bottom)
                       coordinate (-bottom left rear)
                       coordinate (-bottom rear left)
                       coordinate (-rear bottom left)
                       coordinate (-rear left bottom);
       \path (0,0,0) coordinate (-left bottom front)
                     coordinate (-left front bottom)
                     coordinate (-bottom left front)
                     coordinate (-bottom front left)
                     coordinate (-front bottom left)
                     coordinate (-front left bottom);
       \path (#1,#2,0) coordinate (-right top front)
                       coordinate (-right front top)
                       coordinate (-top right front)
                       coordinate (-top front right)
                       coordinate (-front top right)
                       coordinate (-front right top);
       \path (#1,#2,-#3) coordinate (-right top rear)
                         coordinate (-right rear top)
                         coordinate (-top right rear)
                         coordinate (-top rear right)
                         coordinate (-rear top right)
                         coordinate (-rear right top);
       \path (#1,0,-#3) coordinate (-right bottom rear)
                        coordinate (-right rear bottom)
                        coordinate (-bottom right rear)
                        coordinate (-bottom rear right)
                        coordinate (-rear bottom right)
                        coordinate (-rear right bottom);
       \path (#1,0,0) coordinate (-right bottom front)
                      coordinate (-right front bottom)
                      coordinate (-bottom right front)
                      coordinate (-bottom front right)
                      coordinate (-front bottom right)
                      coordinate (-front right bottom);
       \coordinate (-left center) at (0,.5*#2,-.5*#3);
       \coordinate (-right center) at (#1,.5*#2,-.5*#3);
       \coordinate (-top center) at (.5*#1,#2,-.5*#3);
       \coordinate (-bottom center) at (.5*#1,0,-.5*#3);
       \coordinate (-front center) at (.5*#1,.5*#2,0);
       \coordinate (-rear center) at (.5*#1,.5*#2,-#3);
       \coordinate (-center) at (.5*#1,.5*#2,-.5*#3);
       \path (0,#2,-.5*#3) coordinate (-left top center) 
                           coordinate (-top left center);
       \path (.5*#1,#2,-#3) coordinate (-top rear center)
                            coordinate (-rear top center);
       \path (#1,#2,-.5*#3) coordinate (-right top center)
                            coordinate (-top right center);
       \path (.5*#1,#2,0) coordinate (-top front center)
                          coordinate (-front top center);
       \path (0,0,-.5*#3) coordinate (-left bottom center) 
                           coordinate (-bottom left center);
       \path (.5*#1,0,-#3) coordinate (-bottom rear center)
                            coordinate (-rear bottom center);
       \path (#1,0,-.5*#3) coordinate (-right bottom center)
                            coordinate (-bottom right center);
       \path (.5*#1,0,0) coordinate (-bottom front center)
                          coordinate (-front bottom center);
       \path (0,.5*#2,0) coordinate (-left front center) 
                           coordinate (-front left center);
       \path (0,.5*#2,-#3) coordinate (-left rear center)
                            coordinate (-rear left center);
       \path (#1,.5*#2,0) coordinate (-right front center)
                            coordinate (-front right center);
       \path (#1,.5*#2,-#3) coordinate (-right rear center)
                          coordinate (-rear right center);
    \end{scope}
}

\tikzset{
  pics/cuboid/.style = {
    setup code = \tikz@lib@cuboid@setup,
    background code = \tikz@lib@cuboid@draw#1\pgf@stop
  },
  pics/cuboid/.default={1--1--1},
  cuboid/.is family,
  cuboid,
  all faces/.style={fill=white},
  all grids/.style={draw=none},
  front face/.style={},
  front grid/.style={},
  right face/.style={},
  right grid/.style={},
  top face/.style={},
  top grid/.style={},
  edges/.style={},
  hidden edges/.style={draw=none},
  xangle/.initial=0,
  yangle/.initial=90,
  zangle/.initial=210,
  xscale/.initial=1,
  yscale/.initial=1,
  zscale/.initial=0.5
}

\newcommand{\tikzcuboidreset}{
\tikzset{cuboid,
  all faces/.style={fill=white},
  all grids/.style={draw=none},
  front face/.style={},
  front grid/.style={},
  right face/.style={},
  right grid/.style={},
  top face/.style={},
  top grid/.style={},
  edges/.style={},
  hidden edges/.style={draw=none},
  xangle=0,
  yangle=90,
  zangle=210,
  xscale=1,
  yscale=1,
  zscale=0.5
}
}

\newcommand{\tikzcuboidset}{\@ifstar\tikzcuboidset@star\tikzcuboidset@nostar} 
\newcommand{\tikzcuboidset@nostar}[1]{\tikzcuboidreset\tikzset{cuboid,#1}}
\newcommand{\tikzcuboidset@star}[1]{\tikzset{cuboid,#1}}
\makeatother

\usetikzlibrary{angles, quotes}

\makeatletter
\newif\ifnobrackets
\renewcommand\@cite[2]{\ifnobrackets\else[\fi{#1\if@tempswa , #2\fi}\ifnobrackets\else]\fi\nobracketsfalse}

\makeatother

\usepackage{lmodern}
\usetikzlibrary{decorations.pathmorphing}
\tikzset{snake it/.style={decorate, decoration=snake}}

\begin{document}

\title{On computing Goldstein approximate second-order stationary points of structured nonsmooth nonconvex programs
}

\author{
Jiewen GUAN
\thanks{Department of Systems Engineering and Engineering Management, The Chinese University of Hong Kong, Shatin, New Territories, Hong Kong. Email: seemjwguan@gmail.com}
    \and
Anthony Man-Cho SO
\thanks{Department of Systems Engineering and Engineering Management, The Chinese University of Hong Kong, Shatin, New Territories, Hong Kong. Email: manchoso@se.cuhk.edu.hk}
}

\date{\today}

\maketitle

\begin{abstract}
In this paper, we exhibit a randomized first-order algorithm to compute Goldstein approximate second-order stationary points of $L$-smooth functions, 
using tools from
randomized smoothing.
The algorithm has oracle complexity $\widetilde{O}({  n^2}/{\varepsilon^9}+{  n^3}/{\varepsilon^7})$,
where $n=1,2,\ldots$ is the input dimension and $\varepsilon>0$ is the (common) tolerance.
We also present extensions to weakly convex functions and applications to bilevel optimization.
    



\end{abstract}

\section{Introduction}\label{sec:introduction}
For nonconvex functions,
computing second-order stationary points (SOSPs),
or
equivalently,
saddle escape,
is important.
From first principles,
they are the natural next target beyond (first-order) stationary points (FOSPs) in the stationarity hierarchy $\textnormal{FOSPs}\supseteq\textnormal{SOSPs}
\supseteq\cdots$,
thereby offering the promise of
solutions with higher quality.
From a practical point of view,
many problems arising in applications admit no spurious SOSPs (i.e., every such point is globally optimal; e.g., low-rank matrix recovery~\cite{bhojanapalli2016global}, phase retrieval~\cite{sun2018geometric}, and neural collapse~\cite{zhu2021geometric}), which implies that any algorithm with convergence to SOSPs will avoid all spurious FOSPs and converge to global optimality.

It is now folkloric that computing SOSPs of \textit{smooth} nonconvex functions is deterministically tractable; e.g., the cubic Newton methods~\cite{griewank1981modification,nesterov2006cubic,cartis2011adaptive1} and trust-region methods~\cite{conn2000trust,curtis2018concise,jiang2026beyond} already offer finite-time guarantees,
whereas even some first-order methods (e.g., gradient descent~\cite{lee2016gradient}) can converge to SOSPs,
while randomness is often introduced to obtain a finite-time complexity~\cite{ge2015escaping,jin2017escape}.
However, once \textit{nonsmoothness} enters the picture, the situation becomes rather different.
While some recent advances show certain computability of SOSPs for weakly convex functions and beyond under assumptions such as active manifolds and linear tilting (see, e.g.,~\cite{davis2022proximal,bianchi2024stochastic,davis2025active}),
without such assumptions,
our recent work~\cite{guan2026hardness} shows that even for $L$-smooth functions, even a highly permissive target, namely Goldstein SOSPs, already faces fundamental barriers to deterministic computations;
in particular, the complexity of any such algorithm must be $\Omega(n)$, where $n=1,2,\ldots$ is the input dimension.
This naturally leads one to ask what guarantees randomization can offer.

In this paper, while not bypassing the lower bounds, 
we establish 
a
randomized upper bound 
$
    \widetilde{O}\bigl({\Delta L^8 n^2}/{\varepsilon^9}+{\Delta L^6 n^3}/{\varepsilon^7}\bigr)
$
for the problem, where $\Delta>0$ is the initial optimality gap and $\varepsilon>0$ is the (common) tolerance,
via (uniform) randomized smoothing, a technique that has recently been shown to be effective for finding Goldstein FOSPs of Lipschitz functions; see~\cite{lin2022gradient} and also, e.g.,~\cite{kornowski2024algorithm,xia2025revisiting,lei2025subdifferentially} for follow-up works. To be specific, just as~\cite{lin2022gradient} shows that randomized smoothing smooths a Lipschitz function into an $L$-smooth function~\cite[Proposition~2.3]{lin2022gradient} 
and that an FOSP of the smoothed function transfers to a Goldstein FOSP of the original function~\cite[Theorem~3.1]{lin2022gradient},
we show that the same technique likewise smooths an $L$-smooth function into a function with Lipschitz Hessian, together with a similar SOSP transfer property.
With these in place, 
together with some standard estimators for the gradient and Hessian of the smoothed function,
the desired algorithm then follows 
by applying the cubic Newton method to the smoothed function.
In passing, we note that the algorithm is entirely first-order, in a similar spirit to~\cite[Algorithm~1]{lin2022gradient}, thereby obviating the need for cumbersome 
(generalized)
Hessian manipulations.

Building on the algorithm, we then turn to study the computation of Goldstein SOSPs for other structured functions/programs beyond the $L$-smooth setting.
In particular, we propose a new and appropriate notion of Goldstein SOSPs for weakly convex functions via the Moreau envelope~\cite{moreau1965proximite}, which is nontrivial as a naive extension of the notion of Clarke SOSPs~\cite[Theorem~3.1]{hiriart1984generalized} to such functions fails to be necessary for optimality,
and also develop an algorithm for computing such points.
Besides, we apply the two pieces of machinery developed earlier to bilevel programs, which allows us, e.g., to remove the assumption of third-order differentiability for the lower-level function in, e.g.,~\cite{yang2023accelerating,chen2025near,huang2025efficiently} in the nonconvex--strongly-convex (NC--SC) regime while still retaining a meaningful second-order guarantee, and to broaden the applicability of the set-smoothness machinery developed in~\cite{chen2025set}
in the nonconvex--(Polyak--\L{}ojasiewicz) (NC--P\L{}) regime.

The rest of the paper is organized as follows. We first lay out the notation and preliminaries in Section~\ref{sec:notation-prelim}. 
Then, we present the second-order randomized smoothing in Section~\ref{sec:randomized-smoothing}, the randomized upper bound in Section~\ref{sec:computing-Goldstein-SOSPs},
and the extensions to weakly convex functions and applications to bilevel programs in Section~\ref{sec:extensions-applications}.
Finally, we conclude this paper in Section~\ref{sec:conclusions} with some future directions.

\section{Notation and preliminaries}\label{sec:notation-prelim}

\subsection{Notation}
In this paper,
we 
use
lowercase letters (e.g., $x$), boldface lowercase letters (e.g., $\bx=(x_i)$), and boldface capital letters (e.g., $\bX=(x_{i j})$), to denote scalars, vectors, and matrices, respectively. 
We use 
$\|\bx\|$ to denote the 
$\ell_2$-norm of any $\bx\in\R^n$,
and 
$\bbB^n:=\{\bx\in\R^n:\|\bx\|\le 1\}$ and $\mathbb{S}^n:=\{\bx\in\R^n:\|\bx\|= 1\}$ to denote the unit ball and 
unit
sphere in $\R^n$. 
For any $\bA\in\R^{m\times n}$, we write
$\|\bA\|$ for its spectral norm,
$\|\bA\|_1$ for its (entrywise) $\ell_1$-norm,
and $\bA^{\T}$ for its transpose.
When $\bA$ is symmetric, 
$\lambda_{\min}(\bA)$ denotes its smallest eigenvalue.
The 
all-zero vector in $\R^n$ is denoted by $\bd{0}_n$,
the $i$-th standard basis vector in $\R^n$ is denoted by $\be^n_i$,
and the 
identity matrix 
in $\R^{n\times n}$
is denoted by $\bI_n$;
their subscripts and superscripts 
are often omitted as long as there is no ambiguity. 
We use 
$\vee$
to denote
vector concatenation; i.e., $\bx\vee\by=(\bx^{\T},\by^{\T})^{\T}$.
We also adopt $\langle\bullet,\bullet\rangle$ to denote the (Frobenius) inner product; i.e., $\langle\bX,\bY\rangle=\sum_{i=1}^m \sum_{j=1}^n x_{i j}y_{i j}$.
As an aside, $\succeq$ denotes the L\"owner order.


We use $\sum_{j>i}$ and $\sum_{i=1}^{n-1}\sum_{j=i+1}^{n}$ interchangeably as long as $n$ is clear from the context.
Besides,
$\lceil\bullet\rceil$ denotes the ceiling 
operation
and $\log$ denotes the natural logarithm.
We use $\operatorname{Arg}\min$ to denote the set of (global) minimizers of a program, 
and, when the minimizer is unique, 
$\arg\min$ to denote that unique element;
the same convention applies to maximizers as well.
As an aside,
in this paper, we do not distinguish between the maximum and supremum, nor between the minimum and infimum.
For any $\mathbb{X}\subseteq\R^n$, we denote by
$\operatorname{conv}(\bbX)$ its convex hull, by 
$$
    \operatorname{dist}(\by,\bbX):=\min\bigl\{\|\by-\bx\|:\bx\in\bbX\bigr\}
$$ 
the distance from $\by
$ to $\bbX$,
by $\operatorname{vol}(\bbX)$ its volume (i.e., its Lebesgue measure),\footnote{By a slight abuse of notation, we also write $\operatorname{vol}(\bbS^n)$ for the surface area of $\bbS^n$.}
by $\bbX^{\complement}$ its complement,
and by $|\bbX|$ its cardinality.
When $|\bbX|=1$, we identify 
$\bbX$
with its unique element.
We also denote by $\emptyset$ the empty set.
For any $f:\R^n\rightarrow\R\cup\{\pm\infty\}$, we use $\operatorname{dom}(f)$ to denote its
effective domain;
see, e.g.,~\cite[Section~1.A]{rockafellar2009variational}.
For any $\lambda>0$, 
$\bx\in\R^n$,
and
lower semicontinuous $f:\R^n\rightarrow\R$, its Moreau envelope and proximal mapping
are defined by (see, e.g.,~\cite[Definition~1.22]{rockafellar2009variational})
$$
    f_{\lambda}(\bx):=\min_{\by\in\R^n}\mleft(f(\by)+\frac{1}{2\lambda}\|\by-\bx\|^2\mright)
    \quad\text{and}\quad
    \operatorname{Prox}_{\lambda f}(\bx):=\underset{\by\in\R^n}{\operatorname{Arg}\min}\,\mleft(f(\by)+\frac{1}{2\lambda}\|\by-\bx\|^2\mright).
$$
For any $f:\R^n\rightarrow\R$, we say that it is $L$-smooth for an $L\ge 0$ if it is differentiable with $\nabla f$ being $L$-Lipschitz,
whereas 
that it is $\rho$-weakly convex for a $\rho\ge 0$ if $f+2^{-1}\rho\mleft\|\bullet\mright\|^2$ is convex~\cite[Definition~4.1]{vial1983strong}.
In passing, we remark that as weakly convex functions are subdifferentially regular (by, e.g.,~\cite[Example~7.27]{rockafellar2009variational} and~\cite[Exercise~8.20(b)]{rockafellar2009variational}) 
and locally Lipschitz (by, e.g.,~\cite[Proposition~4.4]{vial1983strong}),
their Fr\'echet~\cite[Definition~8.3(a)]{rockafellar2009variational}, Mordukhovich~\cite[Definition~1.77]{mordukhovich2006variational}, and Clarke~\cite[Theorem~2.5.1]{clarke1990optimization} subdifferentials all collapse to the same object (by, e.g.,~\cite[Fact~5]{li2020understanding}), and we thus use $\partial f$ to denote any of them.
For a vector field $\bF:\R^n\rightarrow\R^n$, we use $\operatorname{div}\bF$ to denote its divergence.


We adopt $\bbP$ to denote a probability measure and $\mathbb{E}$ the corresponding expectation, and, when necessary, the underlying law is specified by the subscript; e.g., $\mathbb{E}_{\bu\sim\operatorname{Unif}(\bbB^n)}$, where $\operatorname{Unif}(\bbB^n)$ is the uniform distribution on $\bbB^n$,
and $\sim$ 
means
``is distributed according to.''
We also use $\overset{\operatorname{iid}}{\sim}$ to denote independent and identically distributed sampling,
and $\#$ to denote the pushforward operator.
While asymptotic notation is used extensively in this paper, we refrain from reviewing it in detail, as it adheres primarily to standard conventions in the field. 
All set-valued limit superiors 
are
understood
in the sense of Painlev{\'e}--Kuratowski (see, e.g.,~\cite[Section~5.B]{rockafellar2009variational}), whereas scalar limit superiors are taken in the usual sense; the same convention applies to limit inferiors as well.

\subsection{Generalized differentiation theory and stationarity concepts}\label{sec:diff}
To start, we introduce the (first-order) Clarke and Goldstein subdifferentials.
\begin{definition}\label{def:subdiffs}
Given a 
locally
Lipschitz function $f:\R^n\rightarrow\R$ and a point $\bx\in\R^n$:
\begin{itemize}
    \item The Clarke subdifferential of $f$ at $\bx$ is defined as 
    $$
    \partial_{\dC} f(\bx):=\operatorname{conv}\mleft(\limsup_{\by\rightarrow\bx,\,\by\in\mathbb{D}}\{\nabla f(\by)\}\mright),
    $$
    where $\mathbb{D}\subseteq\R^n$ is the set of points at which $f$ is differentiable;
    see, e.g.,~\cite[Theorem~2.5.1]{clarke1990optimization}.

    \item Given a constant $\delta\ge 0$, the Goldstein $\delta$-subdifferential of $f$ at $\bx$ is defined as
    $$
        \partial_{\delta} f(\bx):=\operatorname{conv}\mleft(\bigcup_{\by\in\bx+\delta\bbB^n}\partial_{\dC} f(\by)\mright);
    $$
    see~\cite[Section~2]{goldstein1977optimization}. 
\end{itemize}
\end{definition}

We note
that for 
a
locally
Lipschitz 
function $f$ and a point $\bx$, the 
sets
$\partial_{\dC} f(\bx)$ and $\partial_{\delta} f(\bx)$ are always nonempty, convex, and compact; see, e.g.,~\cite[Proposition~2.1.2(a)]{clarke1990optimization} and~\cite[Proposition~2.3]{goldstein1977optimization}.
With the first-order constructions in place, we next introduce tools from second-order variational analysis.\footnote{For a comprehensive treatment 
to this subject,
we refer interested readers to Mordukhovich's recent monograph~\cite{mordukhovich2024second}.}
We begin with the Clarke generalized Hessian, defined 
analogously
to $\partial_{\dC} f(\bx)$.

\begin{definition}\label{def:second-subdiffs}
Given an 
$L$-smooth
function $f:\R^n\rightarrow\R$ and a point $\bx\in\R^n$:
    \begin{itemize}
        \item The Clarke generalized Hessian of $f$ at $\bx$ is defined as
        $$
            \partial_{\dC}^2 f(\bx):=\operatorname{conv}\mleft(\limsup_{\by\rightarrow\bx,\,\by\in\bbD}\mleft\{\nabla^2 f(\by)\mright\}\mright)\subseteq\R_{\operatorname{sym}}^{n\times n},
        $$
        where $\mathbb{D}\subseteq\R^n$ is the set of points at which $f$ is twice differentiable; see, e.g.,~\cite[Definition~2.6.1]{clarke1990optimization} and~\cite[Theorem~13.52]{rockafellar2009variational}.
    \end{itemize}
\end{definition}

We note that similar to the first-order constructions,
for an 
$L$-smooth
function $f$ and a point $\bx$, the 
set
$\partial_{\dC}^2 f(\bx)$ is also nonempty, convex, and compact; see, e.g.,~\cite[Proposition~2.6.2(a)]{clarke1990optimization}.
Very naturally, the construction of $\partial_{\dC}^2 f(\bx)$ gives rise to a second-order optimality condition 
that parallels the smooth case; i.e., vanishing gradient and positive semidefinite Hessian.

\begin{lemma}[{\cite[Theorem~3.1]{hiriart1984generalized}}]\label{lma:Clarke-SOSP}
    Suppose that $f:\R^n\rightarrow\R$ is $L$-smooth,
    and let $\bx\in\R^n$ be a local minimizer of $f$. Then, we have
    \begin{equation}\label{eq:Clarke-SOSP}
        \nabla f(\bx)=\bd{0}\quad\text{and}\quad\min_{\bw\in\bbS^n}\max_{\bA\in\partial_{\dC}^2 f(\bx)}\bw^{\T}\bA\bw\ge 0.
    \end{equation}
\end{lemma}

For convenience, we refer to points satisfying (\ref{eq:Clarke-SOSP}) as Clarke SOSPs.
We now turn to the central object in this 
subsection: The Goldstein generalized Hessian.

\begin{definition}
\label{def:Goldstein-2nd-subdiff}
Given a constant $\delta\ge 0$, an 
$L$-smooth
function $f:\R^n\rightarrow\R$, and a point $\bx\in\R^n$, the Goldstein
generalized Hessian
of $f$ at $\bx$ is 
defined as
    $$
        \partial_{\delta}^2 f(\bx)
        :=\operatorname{conv}\mleft(\bigcup_{\by\in\bx+\delta\bbB^n}\partial_{\dC}^2 f(\by)
        \mright).
    $$
\end{definition}

In other words, $\partial_{\delta}^2 f(\bx)$ is the smallest convex set that contains all possible 
$\partial_{\dC}^2 f(\by)$
from points $\by$ with distance at most $\delta$ from $\bx$, 
in direct analogy with
$\partial_{\delta} f(\bx)$.
This explains
why $\partial_{\delta}^2 f(\bx)$ is referred to as ``Goldstein''.
We also note that,
essentially, 
$\partial_{\delta}^2 f(\bx)$
is closely related to, yet distinct from, 
the 
second-order $\varepsilon$-jet
introduced by Gebken;
see~\cite[Definition~1]{gebken2022using}.

Next,
we introduce
the stationarity concept that 
$\partial_{\delta}^2 f(\bx)$ induces. 

\begin{definition}
\label{def:Goldstein-SOSP}
    Given constants $\varepsilon_1,\varepsilon_2,\delta\ge 0$, an 
    $L$-smooth
    function $f:\R^n\rightarrow\R$, 
    and a point $\bx\in\R^n$, we say that $\bx$ is an $(\varepsilon_1,\varepsilon_2,\delta)$-Goldstein approximate 
    SOSP
    of $f$ if 
    $$
        \|\nabla f(\bx)\|\le\varepsilon_1\quad\text{and}\quad\min_{\bw\in\bbS^n}\max_{\bA\in\partial_{\delta}^2 f(\bx)}\bw^{\T}\bA\bw\ge-\varepsilon_2.
    $$
\end{definition}


In closing, 
we introduce second subderivatives, which, while
not serving as our algorithmic target, will play a crucial role in our analysis.

\begin{definition}
Given a 
function $f:\R^n\rightarrow\R$, a point $\bx\in\R^n$, and vectors $\bv,\bw\in
\R^n$:
    \begin{itemize}    
        \item The second subderivative of $f$ at $\bx$ for $\bv$
        and $\bw
        $
        is
        defined as
        $$
            \dd^2{f(\bx;\bv)}(\bw):=\liminf_{\substack{t\searrow 0,\, \bw^{\prime}\rightarrow\bw}}\frac{f(\bx+t\bw^{\prime})-f(\bx)-t\cdot\bv^{\T}\bw^{\prime}}{\frac{1}{2}t^2};
        $$
        see, e.g.,~\cite[Definition~13.3]{rockafellar2009variational}.
    \end{itemize}
\end{definition}

As in Lemma~\ref{lma:Clarke-SOSP}, second subderivatives likewise give rise to a second-order optimality condition,
whose solutions we refer to as Rockafellar SOSPs.

\begin{lemma}[{\cite[Theorem~13.24(a)]{rockafellar2009variational}}]
    If
    $\bx\in\R^n$ is a local minimizer of 
    $f:\R^n\rightarrow\R$, 
    then\footnote{Here, the subdifferential is understood in the sense of Mordukhovich; see, e.g.,~\cite[Definition~1.77]{mordukhovich2006variational}.}
    $$
        \bd{0}\in\partial f(\bx)
        \quad\text{and}\quad
        \dd^2{f(\bx;\bd{0})}(\bw)\ge 0
        \quad\text{for all $\bw\in\R^n$}.        
    $$
\end{lemma}

\section{Second-order randomized smoothing}\label{sec:randomized-smoothing}

\subsection{From differentiability to twice differentiability}
The following theorem is the main result of this section and underpins everything that follows.

\begin{theorem}\label{thm:smoothing}
    Let $L,\sigma>0$ be arbitrary and $f:\R^n\rightarrow\R$ be $L$-smooth. For $f_{\sigma}:\R^n\rightarrow\R$ given by
    $$
        f_{\sigma}(\bx):=\mathbb{E}_{\bu\sim\operatorname{Unif}(\bbB^n)}[f(\bx+\sigma\bu)]\quad\text{for all $\bx\in\R^n$},
    $$
    the following statements hold.
    \begin{itemize}
        \item The function $f_{\sigma}:\R^n\rightarrow\R$ is twice differentiable.

        \item For all $\bx,\by\in\R^n$, the mapping $\nabla f_{\sigma}:\R^n\rightarrow\R^n$ satisfies that
        $$
            \nabla f_{\sigma}(\bx)=\mathbb{E}_{\bu\sim\operatorname{Unif}(\bbB^n)}[\nabla f(\bx+\sigma\bu)]\quad\text{and}\quad\|\nabla f_{\sigma}(\bx)-\nabla f_{\sigma}(\by)\|\le L \|\bx-\by\|.
        $$

        \item For all $\bx,\by\in\R^n$, the mapping $\nabla^2 f_{\sigma}:\R^n\rightarrow\R_{\operatorname{sym}}^{n\times n}$ satisfies that
        $$
            \nabla^2 f_{\sigma}(\bx)=\mathbb{E}_{\bu\sim\operatorname{Unif}(\bbB^n)}[\nabla^2 f(\bx+\sigma\bu)]\quad\text{and}\quad\|\nabla^2 f_{\sigma}(\bx)-\nabla^2 f_{\sigma}(\by)\|
            \le\frac{c L \sqrt{n}}{\sigma}\|\bx-\by\|,
        $$
        where $c>0$ is a fixed constant.
    \end{itemize}
\end{theorem}

\begin{proof}
    We shall first deal with the differentiability of $f_{\sigma}$. To this end, for some given $\bx\in\R^n$, let us consider the functions $f_{\bx,i}:(-1,1)\times\bbB^n\rightarrow\R$ for $i=1,\ldots,n$ defined by $f_{\bx,i}(t,\bu):=f(\bx+\sigma\bu+t\be_i)$. First, by~\cite[Theorem~4.14]{rudin1976principles}, it follows that for any $t\in(-1,1)$, there exists some $M>0$, such that $|f_{\bx,i}(t,\bu)|\le M$ for all $\bu\in\bbB^n$. 
    As a result, for all such $t$,
    we have $\mathbb{E}_{\bu\sim\operatorname{Unif}(\bbB^n)}[|f_{\bx,i}(t,\bu)|]\le M<\infty$; i.e., 
    $f_{\bx,i}$ is Lebesgue-integrable w.r.t.\ $\bu$. Besides, as $f$ is $L$-smooth as assumed, we know that $\partial f_{\bx,i}(t,\bu)/\partial t$ exists and equals $\partial f(\bx+\sigma\bu+t\be_i)/\partial x_i$ for all $t\in(-1,1)$ and $\bu\in\bbB^n$.
    Moreover, by~\cite[Theorem~4.14]{rudin1976principles} again, 
    there further exists some $M^{\prime}>0$ such that $|\partial f_{\bx,i}(t,\bu)/\partial t|\le M^{\prime}$ for all $t\in(-1,1)$ and $\bu\in\bbB^n$. It then follows 
    from~\cite[Theorem~2.27(b)]{folland1999real} that for all $i=1,\ldots,n$, the function $t\mapsto f_{\sigma}(\bx+t\be_i)=\mathbb{E}_{\bu\sim\operatorname{Unif}(\bbB^n)}[f_{\bx,i}(t,\bu)]$ is differentiable at $0$ with 
    $$
        \frac{\partial f_{\sigma}(\bx)}{\partial x_i}=\bigl(t\mapsto f_{\sigma}(\bx+t\be_i)\bigr)^{\prime}(0)=
        \mathbb{E}_{\bu\sim\operatorname{Unif}(\bbB^n)}\mleft[\frac{\partial f_{\bx,i}}{\partial t}(0,\bu)\mright]=\mathbb{E}_{\bu\sim\operatorname{Unif}(\bbB^n)}\mleft[\frac{\partial f}{\partial x_i}(\bx+\sigma\bu)\mright];
    $$
    i.e., $\nabla f_{\sigma}(\bx)=\mathbb{E}_{\bu\sim\operatorname{Unif}(\bbB^n)}[\nabla f(\bx+\sigma\bu)]$.
    It remains to show that $\partial f_{\sigma}(\bx)/\partial x_i$ for $i=1,\ldots,n$ are all Lipschitz continuous. 
    
    Indeed, for all $\bx,\by\in\R^n$, we have 
    $$
    \begin{aligned}
        \left\|\begin{pmatrix}
            \partial f_{\sigma}(\bx)/\partial x_1 \\
            \vdots \\
            \partial f_{\sigma}(\bx)/\partial x_n
        \end{pmatrix}-\begin{pmatrix}
            \partial f_{\sigma}(\by)/\partial x_1 \\
            \vdots \\
            \partial f_{\sigma}(\by)/\partial x_n
        \end{pmatrix}\right\|&=\left\|\mathbb{E}_{\bu\sim\operatorname{Unif}(\bbB^n)}\bigl[\nabla f (\bx+\sigma\bu)\bigr]-\mathbb{E}_{\bu\sim\operatorname{Unif}(\bbB^n)}\bigl[\nabla f (\by+\sigma\bu)\bigr]\right\|\\
        &\le\mathbb{E}_{\bu\sim\operatorname{Unif}(\bbB^n)}\bigl[\|\nabla f (\bx+\sigma\bu)-\nabla f (\by+\sigma\bu)\|\bigr]\le L \|\bx-\by\|;
    \end{aligned}
    $$
    i.e., $\nabla f_{\sigma}$ is $L$-Lipschitz, as desired. Here, the first inequality follows from Jensen's inequality; see, e.g.,~\cite[Theorem~1.6.2]{durrett2019probability}. 

    We next turn to the twice differentiability of $f_{\sigma}$. Given $\bx\in\R^n$, we see that
    $$
    \begin{aligned}
        \frac{1}{t}\left(\frac{\partial f_{\sigma}}{\partial x_j}(\bx+t\be_i)-\frac{\partial f_{\sigma}}{\partial x_j}(\bx)\right)&=\frac{1}{t}\left(\mathbb{E}_{\bu\sim\operatorname{Unif}(\bbB^n)}\mleft[\frac{\partial f}{\partial x_j}(\bx+\sigma\bu+t\be_i)\mright]-\mathbb{E}_{\bu\sim\operatorname{Unif}(\bbB^n)}\mleft[\frac{\partial f}{\partial x_j}(\bx+\sigma\bu)\mright]\right)\\
        &=\int_{\bbB^n}\frac{1}{t}\left(\frac{\partial f}{\partial x_j}(\bx+\sigma\bu+t\be_i)-\frac{\partial f}{\partial x_j}(\bx+\sigma\bu)\right)\bbP_{\bu\sim\operatorname{Unif}(\bbB^n)}(\dd\bu)
    \end{aligned}
    $$
    for all $j=1,\ldots,n$. In what follows, we shall take limits on both sides and pass the limit on the right-hand side by applying the dominated convergence theorem (DCT); see, e.g.,~\cite[Theorem~2.24]{folland1999real}.
    To begin, as $f$ is $L$-smooth, we know from Rademacher's theorem~\cite[Theorem~9.60]{rockafellar2009variational} that $\nabla f$ is a.e.\ differentiable. Hence, for all $i=1,\ldots,n$, 
    the limit
    $$
        \frac{\partial^2 f}{\partial x_i\partial x_j}(\bx+\sigma\bu)=\lim_{t\rightarrow 0}\frac{1}{t}\left(\frac{\partial f}{\partial x_j}(\bx+\sigma\bu+t\be_i)-\frac{\partial f}{\partial x_j}(\bx+\sigma\bu)\right)
    $$
    exists for $\bbP_{\bu\sim\operatorname{Unif}(\bbB^n)}$-a.e.\ $\bu\in\bbB^n$. 
    This, together with Heine's theorem~\cite[Proposition~1]{zorich2015mathematical}, further implies that for all sequences $t_1,t_2,\ldots$ with $\lim_{k\rightarrow\infty} t_k=0$ and all such $\bu\in\bbB^n$, the limit
    $$
        \lim_{k\rightarrow\infty}\frac{1}{t_k}\left(\frac{\partial f}{\partial x_j}(\bx+\sigma\bu+t_k\be_i)-\frac{\partial f}{\partial x_j}(\bx+\sigma\bu)\right)
    $$
    exists and equals $\partial^2 f(\bx+\sigma\bu)/\partial x_i\partial x_j$.
    Besides, 
    for all $t\neq 0$, we have
    $$
        \left|\frac{1}{t}\left(\frac{\partial f}{\partial x_j}(\bx+\sigma\bu+t\be_i)-\frac{\partial f}{\partial x_j}(\bx+\sigma\bu)\right)\right|\le\frac{1}{|t|}\left\|\nabla f(\bx+\sigma\bu+t\be_i)-\nabla f(\bx+\sigma\bu)\right\|\le L,
    $$
    and thus
    $$
        \int_{\bbB^n}\left|\frac{1}{t}\left(\frac{\partial f}{\partial x_j}(\bx+\sigma\bu+t\be_i)-\frac{\partial f}{\partial x_j}(\bx+\sigma\bu)\right)\right|\bbP_{\bu\sim\operatorname{Unif}(\bbB^n)}(\dd\bu)<\infty.
    $$
    It then follows from the DCT that for all sequences $t_1,t_2,\ldots$ with $\lim_{k\rightarrow\infty} t_k=0$,
    $$
    \begin{aligned}
        &\lim_{k\rightarrow\infty}\frac{1}{t_k}\left(\frac{\partial f_{\sigma}}{\partial x_j}(\bx+t_k\be_i)-\frac{\partial f_{\sigma}}{\partial x_j}(\bx)\right)\\
        =&\lim_{k\rightarrow\infty}\int_{\bbB^n}\frac{1}{t_k}\left(\frac{\partial f}{\partial x_j}(\bx+\sigma\bu+t_k\be_i)-\frac{\partial f}{\partial x_j}(\bx+\sigma\bu)\right)\bbP_{\bu\sim\operatorname{Unif}(\bbB^n)}(\dd\bu)\\
        =&\int_{\bbB^n}\lim_{k\rightarrow\infty}\frac{1}{t_k}\left(\frac{\partial f}{\partial x_j}(\bx+\sigma\bu+t_k\be_i)-\frac{\partial f}{\partial x_j}(\bx+\sigma\bu)\right)\bbP_{\bu\sim\operatorname{Unif}(\bbB^n)}(\dd\bu)\\
        =&\int_{\bbB^n}\frac{\partial^2 f}{\partial x_i\partial x_j}(\bx+\sigma\bu)\bbP_{\bu\sim\operatorname{Unif}(\bbB^n)}(\dd\bu)=\mathbb{E}_{\bu\sim\operatorname{Unif}(\bbB^n)}\mleft[\frac{\partial^2 f}{\partial x_i\partial x_j}(\bx+\sigma\bu)\mright].
    \end{aligned}
    $$
    Hence, we have by Heine's theorem~\cite[Proposition~1]{zorich2015mathematical} again that for all $i,j=1,\ldots,n$,
    $$
        \frac{\partial^2 f_{\sigma}}{\partial x_i\partial x_j}(\bx)=\lim_{t\rightarrow 0}\frac{1}{t}\left(\frac{\partial f_{\sigma}}{\partial x_j}(\bx+t\be_i)-\frac{\partial f_{\sigma}}{\partial x_j}(\bx)\right)=\mathbb{E}_{\bu\sim\operatorname{Unif}(\bbB^n)}\mleft[\frac{\partial^2 f}{\partial x_i\partial x_j}(\bx+\sigma\bu)\mright];
    $$ 
    i.e., $\nabla^2 f_{\sigma}(\bx)=\mathbb{E}_{\bu\sim\operatorname{Unif}(\bbB^n)}[\nabla^2 f(\bx+\sigma\bu)]$.
    It remains to show that $\partial^2 f_{\sigma}(\bx)/\partial x_i\partial x_j$ for $i,j=1,\ldots,n$ are all Lipschitz continuous in the spectral norm. 

    To this end, for each direction $\bdd\in\bbS^n$, let us consider the marginal function $h(\bullet;\bdd):\R^n\rightarrow\R$ defined by $h(\bx;\bdd):=\bdd^{\T}\nabla f(\bx)$; as a sanity check, we have $\nabla h(\bx;\bdd)=\nabla^2 f(\bx)\bdd$ a.e. (In passing, we note that as $f$ is $L$-smooth and thus prox-regular (see, e.g.,~\cite[Proposition~13.34]{rockafellar2009variational}), it follows from~\cite[Corollary~13.42]{rockafellar2009variational} that $\nabla^2 f(\bx)$ must be symmetric at every point where it exists.) As $\nabla f$ is $L$-Lipschitz as assumed, it follows from~\cite[Exercise~9.9]{rockafellar2009variational} that all the functions $h(\bullet;\bdd)$ are $L$-Lipschitz. As a result, we know from~\cite[Proposition~2.3]{lin2022gradient} that for every $\bdd\in\bbS^n$, the function $h_{\sigma}(\bullet;\bdd):\R^n\rightarrow\R$ defined by $h_{\sigma}(\bx;\bdd):=\mathbb{E}_{\bu\sim\operatorname{Unif}(\bbB^n)}[h(\bx+\sigma\bu;\bdd)]$ is ${c L \sqrt{n}}/{\sigma}$-smooth for some fixed $c>0$.
    Besides, by definition, we have
    $$
        h_{\sigma}(\bx;\bdd)
        =\mathbb{E}_{\bu\sim\operatorname{Unif}(\bbB^n)}[\bdd^{\T}\nabla f(\bx+\sigma\bu)]=\bdd^{\T}\mathbb{E}_{\bu\sim\operatorname{Unif}(\bbB^n)}[\nabla f(\bx+\sigma\bu)]=\bdd^{\T}\nabla f_{\sigma}(\bx),
    $$
    which, together with the twice differentiability of $f_{\sigma}$ shown earlier, further implies that $h_{\sigma}(\bullet;\bdd)$ is differentiable with $\nabla h_{\sigma}(\bx;\bdd)=\nabla^2 f_{\sigma}(\bx)\bdd$.
    As a result, we have
    $$
    \begin{aligned}
        \|\nabla^2 f_{\sigma}(\bx)-\nabla^2 f_{\sigma}(\by)\|&=\max_{\bdd\in\bbS^n}\|\nabla^2 f_{\sigma}(\bx)\bdd-\nabla^2 f_{\sigma}(\by)\bdd\|\\
        &=\max_{\bdd\in\bbS^n}\|\nabla h_{\sigma}(\bx;\bdd)-\nabla h_{\sigma}(\by;\bdd)\|\le\frac{c L \sqrt{n}}{\sigma}\|\bx-\by\|\quad\text{for all $\bx,\by\in\R^n$},
    \end{aligned}
    $$
    as desired. 
\end{proof}


For further developments in randomized smoothing,
we refer interested readers to, e.g.,~\cite{duchi2012randomized,yousefian2012stochastic,nesterov2017random}.

\subsection{Gradient and Hessian estimators}
With Theorem~\ref{thm:smoothing} in place, we next investigate how to estimate the gradient and Hessian of $f_{\sigma}$.



    

\subsubsection{Gradient estimator}\label{sec:gradient-estimator}
Recall that $\nabla f_{\sigma}(\bx)=\mathbb{E}_{\bu\sim\operatorname{Unif}(\bbB^n)}[\nabla f(\bx+\sigma\bu)]$, and thus 
$$
    \nabla f_{\sigma}(\bx)=\mathbb{E}_{\bu\sim\operatorname{Unif}(\bbB^n)}\mleft[\frac{\nabla f(\bx+\sigma\bu)+\nabla f(\bx-\sigma\bu)}{2}\mright],
$$
as $\bu$ and $-\bu$ are equal in distribution. This immediately suggests the following estimator.

\begin{proposition}\label{prop:gradient-estimator}
    Let $L,\sigma>0$, $\bx\in\R^n$, and $b=1,2,\ldots$ be arbitrary, $f:\R^n\rightarrow\R$ be $L$-smooth, and 
    $$
        \widehat{\bg}_b(\bx):=\frac{1}{b}\sum_{j=1}^b \frac{\nabla f(\bx+\sigma\bu_j)+\nabla f(\bx-\sigma\bu_j)}{2},\quad\text{where $\bu_1,\ldots,\bu_b\sim\operatorname{Unif}(\bbB^n)$ independently}.
    $$
    We have $\widehat{\bg}_b(\bx)$ is unbiased and satisfies that
    $$
        \mathbb{P}\mleft\{\|\widehat{\bg}_b(\bx)-\nabla f_{\sigma}(\bx)\|\ge t\mright\}\le 2\exp\mleft(-\frac{b t^2}{2 L^2\sigma^2+ 4 L\sigma t/3}\mright)\quad\text{for all $t>0$}.
    $$
\end{proposition}


\begin{proof}
    We have
    $$
    \begin{aligned}
        \mathbb{E}_{\bu_1,\ldots,\bu_b\overset{\operatorname{iid}}{\sim}\operatorname{Unif}(\bbB^n)}[\widehat{\bg}_b(\bx)]&=
        \mathbb{E}_{\bu_1,\ldots,\bu_b\overset{\operatorname{iid}}{\sim}\operatorname{Unif}(\bbB^n)}\mleft[\frac{1}{b}\sum_{j=1}^b \frac{\nabla f(\bx+\sigma\bu_j)+\nabla f(\bx-\sigma\bu_j)}{2}\mright]\\
        &=\frac{1}{b}\sum_{j=1}^b\mathbb{E}_{\bu_j\sim\operatorname{Unif}(\bbB^n)}\mleft[\frac{\nabla f(\bx+\sigma\bu_j)+\nabla f(\bx-\sigma\bu_j)}{2}\mright]=\nabla f_{\sigma}(\bx);
    \end{aligned}
    $$
    i.e., $\widehat{\bg}_b(\bx)$ is unbiased. 
    
    Let us write $\widehat{\bg}_b(\bx)=\frac{1}{b}\sum_{j=1}^b \widehat{\bg}_{b,j}(\bx)$. We shall apply the vector Bernstein inequality (see, e.g.,~\cite[Theorem~32]{cabannes2021fast} and~\cite{pinelis1986remarks}) to obtain the concentration bound. To this end, we first observe that
    $$
    \begin{aligned}
        \|\widehat{\bg}_{b,j}(\bx)-\nabla f_{\sigma}(\bx)\|\le\|\widehat{\bg}_{b,j}(\bx)-\nabla f(\bx)\|+\|\nabla f(\bx)-\nabla f_{\sigma}(\bx)\|.
    \end{aligned}
    $$
    On the one hand,
    \begin{equation}\label{eq:bound-on-gbj-dfx}
    \begin{aligned}
        \|\widehat{\bg}_{b,j}(\bx)-\nabla f(\bx)\|=\left\|\frac{\nabla f(\bx+\sigma\bu_j)+\nabla f(\bx-\sigma\bu_j)}{2}-\nabla f(\bx)\right\|\le L\|\sigma\bu_j\|\le L\sigma.
    \end{aligned}
    \end{equation}
    On the other hand,
    $$
    \begin{aligned}
        \|\nabla f(\bx)-\nabla f_{\sigma}(\bx)\|&=\|\nabla f(\bx)-\mathbb{E}_{\bu\sim\operatorname{Unif}(\bbB^n)}[\nabla f(\bx+\sigma\bu)]\|\\
        &=\|\mathbb{E}_{\bu\sim\operatorname{Unif}(\bbB^n)}[\nabla f(\bx)-\nabla f(\bx+\sigma\bu)]\|\\
        &\le\mathbb{E}_{\bu\sim\operatorname{Unif}(\bbB^n)}[\|\nabla f(\bx)-\nabla f(\bx+\sigma\bu)\|]\le\mathbb{E}_{\bu\sim\operatorname{Unif}(\bbB^n)}[L\|\sigma\bu\|]\le L\sigma.
    \end{aligned}
    $$
    By putting the above two pieces together, it follows that
    $$
        \|\widehat{\bg}_{b,j}(\bx)-\nabla f_{\sigma}(\bx)\|\le 2 L\sigma.
    $$
    Besides, as the mean minimizes the mean-square error, we have
    $$
    \begin{aligned}
        \mathbb{E}_{\bu_j\sim\operatorname{Unif}(\bbB^n)}[\|\widehat{\bg}_{b,j}(\bx)-\nabla f_{\sigma}(\bx)\|^2]\le\mathbb{E}_{\bu_j\sim\operatorname{Unif}(\bbB^n)}[\|\widehat{\bg}_{b,j}(\bx)-\nabla f(\bx)\|^2]\le
        L^2\sigma^2,
    \end{aligned}
    $$
    where the last inequality follows from (\ref{eq:bound-on-gbj-dfx}). As a result, the vector Bernstein inequality implies that
    $$
        \mathbb{P}\mleft\{\left\|\sum_{j=1}^b \left(\widehat{\bg}_{b,j}(\bx)-\nabla f_{\sigma}(\bx)\right)\right\|\ge t\mright\}\le 2\exp\mleft(-\frac{t^2}{2 b L^2\sigma^2+ 4 t L\sigma/3}\mright)\quad\text{for all $t>0$},
    $$
    from which the desired concentration bound immediately follows. 
    We are done.
\end{proof}

\subsubsection{Hessian estimator}
While $\nabla^2 f_{\sigma}(\bx)=\mathbb{E}_{\bu\sim\operatorname{Unif}(\bbB^n)}[\nabla^2 f(\bx+\sigma\bu)]$ also 
suggests a second-order estimator for $\nabla^2 f_{\sigma}(\bx)$,
as $\nabla^2 f(\bx)$ may not be available or can be expensive in practice, and to match the order 
of information used by
the gradient estimator $\widehat{\bg}_b(\bx)$ in Section~\ref{sec:gradient-estimator}, 
we resort to a first-order Hessian estimator. To this end, we 
begin with
a
technical tool connecting
$\nabla^2 f_{\sigma}(\bx)$ with $\nabla f(\bx)$ instead,
which can also be taken as a 
second-order parallel of~\cite[Lemma~2.1]{flaxman2005online}.

\begin{proposition}\label{prop:divergence}    
    Let $L,\sigma>0$ and $\bx\in\R^n$ be arbitrary, and $f:\R^n\rightarrow\R$ be $L$-smooth. We have
    $$
        \nabla^2 f_{\sigma}(\bx)=\frac{n}{\sigma}\mathbb{E}_{\bs\sim\operatorname{Unif}(\bbS^n)}[\nabla f(\bx+\sigma\bs)\bs^{\T}].
    $$
\end{proposition}

\begin{proof}
    Given $\bv\in\R^n$, let us consider the vector field $\bF:\R^n\rightarrow\R^n$ given by $\bF(\bz):=f(\bz)\bv$. It holds that $\operatorname{div}\bF(\bz)=\bv^{\T}\nabla f(\bz)$. As a result, the divergence theorem (see, e.g.,~\cite[Section~5.9]{folland2024advanced}) implies that
    $$
        \int_{\bx+\sigma\bbB^n}\bv^{\T}\nabla f(\bz)\dd\bz=\int_{\bx+\sigma\bbB^n}\operatorname{div}\bF(\bz)\dd\bz=\frac{1}{\sigma}\int_{\bx+\sigma\bbS^n}f(\bz)\bv^{\T}(\bz-\bx) S(\dd\bz),
    $$
    where $S$ is the surface measure on $\bx+\sigma\bbS^n$; i.e., the restriction of the Hausdorff measure to $\bx+\sigma\bbS^n$ (see, e.g.,~\cite[Section~3.3.4(D)]{evans2015measure}). As $\bv$ is arbitrary, this further implies that
    $$
        \sigma^n\operatorname{vol}(\bbB^n)\nabla f_{\sigma}(\bx)=\int_{\bx+\sigma\bbB^n}\nabla f(\bz)\dd\bz=\frac{1}{\sigma}\int_{\bx+\sigma\bbS^n}f(\bz)(\bz-\bx) S(\dd\bz).
    $$
    Consider the change of variables $\bz=\bx+\sigma\bs$ where $\bs\in\bbS^n$. Then, $S(\dd\bz)=\sigma^{n-1} S(\dd\bs)$; see, e.g.,~\cite[Theorem~2.2(iv--v)]{evans2015measure}. As a result, it follows from~\cite[Theorem~3.6.1]{bogachev2007measure} that
    $$
        \int_{\bx+\sigma\bbS^n}f(\bz)(\bz-\bx) S(\dd\bz)=\sigma^n\int_{\bbS^n}f(\bx+\sigma\bs)\bs S(\dd\bs),
    $$
    and thus
    $$
    \begin{aligned}
        \nabla f_{\sigma}(\bx)&=\frac{1}{\sigma\operatorname{vol}(\bbB^n)}\int_{\bbS^n}f(\bx+\sigma\bs)\bs S(\dd\bs)\\
        &=\frac{n}{\sigma\operatorname{vol}(\bbS^n)}\int_{\bbS^n}f(\bx+\sigma\bs)\bs S(\dd\bs)=\frac{n}{\sigma}\int_{\bbS^n}f(\bx+\sigma\bs)\bs\bbP_{\bs\sim\operatorname{Unif}(\bbS^n)}(\dd\bs).
    \end{aligned}
    $$
    This, together with~\cite[Theorem~2.27(b)]{folland1999real},
    further implies that
    $$
        \nabla^2 f_{\sigma}(\bx)=\frac{n}{\sigma}\int_{\bbS^n}\bs\nabla f(\bx+\sigma\bs)^{\T}\bbP_{\bs\sim\operatorname{Unif}(\bbS^n)}(\dd\bs)=\frac{n}{\sigma}\mathbb{E}_{\bs\sim\operatorname{Unif}(\bbS^n)}[\bs\nabla f(\bx+\sigma\bs)^{\T}].
    $$
    As $\nabla^2 f_{\sigma}(\bx)$ is symmetric, the desired result then follows by taking transpose on both sides.
\end{proof}

As $\bs$ and $-\bs$ equal in distribution, Proposition~\ref{prop:divergence} further implies that
$$
    \nabla^2 f_{\sigma}(\bx)=\mathbb{E}_{\bs\sim\operatorname{Unif}(\bbS^n)}\mleft[\frac{n}{2\sigma}\bigl(\nabla f(\bx+\sigma\bs)-\nabla f(\bx-\sigma\bs)\bigr)\bs^{\T}\mright],
$$
which motivates the following estimator.

\begin{proposition}\label{prop:Hessian-estimator}
    Let $L,\sigma>0$, $\bx\in\R^n$, and $b=1,2,\ldots$ be arbitrary, $f:\R^n\rightarrow\R$ be $L$-smooth, and 
    $$
        \widehat{\bH}_b(\bx):=\frac{1}{b}\sum_{j=1}^b \operatorname{Sym}\mleft(\frac{n}{2\sigma}\bigl(\nabla f(\bx+\sigma\bs_j)-\nabla f(\bx-\sigma\bs_j)\bigr)\bs_j^{\T}\mright),
    $$
    where $\bs_1,\ldots,\bs_b\sim\operatorname{Unif}(\bbS^n)$ independently
    and $\operatorname{Sym}(\bA):=(\bA+\bA^{\T})/2$.
    We have $\widehat{\bH}_b(\bx)$ is unbiased and satisfies that
    $$
        \mathbb{P}\mleft\{\bigl\|\widehat{\bH}_b(\bx)-\nabla^2 f_{\sigma}(\bx)\bigr\|\ge t\mright\}\le 2 n\exp\mleft(-\frac{b t^2/2}{(n+1)^2 L^2+(n+1) L t/3}\mright)\quad\text{for all $t>0$}.
    $$
\end{proposition}


\begin{proof}
    We have
    $$
    \begin{aligned}
        \mathbb{E}_{\bs_1,\ldots,\bs_b\overset{\operatorname{iid}}{\sim}\operatorname{Unif}(\bbS^n)}[\widehat{\bH}_b(\bx)]&=\mathbb{E}_{\bs_1,\ldots,\bs_b\overset{\operatorname{iid}}{\sim}\operatorname{Unif}(\bbS^n)}\mleft[\frac{1}{b}\sum_{j=1}^b \operatorname{Sym}\mleft(\frac{n}{2\sigma}\bigl(\nabla f(\bx+\sigma\bs_j)-\nabla f(\bx-\sigma\bs_j)\bigr)\bs_j^{\T}\mright)\mright]\\
        &=\frac{1}{b}\sum_{j=1}^b \operatorname{Sym}\mleft(\mathbb{E}_{\bs_j\sim\operatorname{Unif}(\bbS^n)}\mleft[\frac{n}{2\sigma}\bigl(\nabla f(\bx+\sigma\bs_j)-\nabla f(\bx-\sigma\bs_j)\bigr)\bs_j^{\T}\mright]\mright)\\
        &=\frac{1}{b}\sum_{j=1}^b \operatorname{Sym}\bigl(\nabla^2 f_{\sigma}(\bx)\bigr)=\operatorname{Sym}\bigl(\nabla^2 f_{\sigma}(\bx)\bigr)=\nabla^2 f_{\sigma}(\bx);
    \end{aligned}
    $$
    i.e., $\widehat{\bH}_b(\bx)$ is unbiased.

    Let us write $\widehat{\bH}_b(\bx)=\frac{1}{b}\sum_{j=1}^b \widehat{\bH}_{b,j}(\bx)$. We shall apply the matrix Bernstein inequality (see, e.g.,~\cite[Theorem~5.4.1]{vershynin2018high}) to obtain the concentration bound. To this end, we first observe that
    $$
    \begin{aligned}
        \bigl\|\widehat{\bH}_{b,j}(\bx)-\nabla^2 f_{\sigma}(\bx)\bigr\|\le \bigl\|\widehat{\bH}_{b,j}(\bx)\bigr\|+\bigl\|\nabla^2 f_{\sigma}(\bx)\bigr\|.
    \end{aligned}
    $$
    On the one hand,
    \begin{equation}
    \begin{aligned}
        \bigl\|\widehat{\bH}_{b,j}(\bx)\bigr\|&=\mleft\|\operatorname{Sym}\mleft(\frac{n}{2\sigma}\bigl(\nabla f(\bx+\sigma\bs_j)-\nabla f(\bx-\sigma\bs_j)\bigr)\bs_j^{\T}\mright)\mright\|\\
        &\le \mleft\|\frac{n}{2\sigma}\bigl(\nabla f(\bx+\sigma\bs_j)-\nabla f(\bx-\sigma\bs_j)\bigr)\bs_j^{\T}\mright\|\\
        &=\frac{n}{2\sigma}\bigl\|\nabla f(\bx+\sigma\bs_j)-\nabla f(\bx-\sigma\bs_j)\bigr\|\cdot\|\bs_j\|\le \frac{n}{2\sigma} L \|2\sigma\bs_j\|=n L,
    \end{aligned}
    \end{equation}
    where we have used the fact that $\|\bA\|=\|\bA^{\T}\|$ for all matrices $\bA$ (which may be rectangular and not symmetric).
    On the other hand, as $f_{\sigma}$ is $L$-smooth and twice continuously differentiable, it follows from~\cite[Theorem~9.7]{rockafellar2009variational} that $\bigl\|\nabla^2 f_{\sigma}(\bx)\bigr\|\le L$ for all $\bx\in\R^n$. By putting the above two pieces together, it follows that
    $$
        \bigl\|\widehat{\bH}_{b,j}(\bx)-\nabla^2 f_{\sigma}(\bx)\bigr\|\le (n+1) L.
    $$
    Besides, by Jensen's inequality (see, e.g.,~\cite[Theorem~1.6.2]{durrett2019probability}),
    we have
    $$
    \begin{aligned}
        \left\|\mathbb{E}_{\bs_j\sim\operatorname{Unif}(\bbS^n)}\mleft[\bigl(\widehat{\bH}_{b,j}(\bx)-\nabla^2 f_{\sigma}(\bx)\bigr)^2\mright]\right\|&\le 
        \mathbb{E}_{\bs_j\sim\operatorname{Unif}(\bbS^n)}\mleft[\left\|\bigl(\widehat{\bH}_{b,j}(\bx)-\nabla^2 f_{\sigma}(\bx)\bigr)^2\right\|\mright]\\
        &\le 
        \mathbb{E}_{\bs_j\sim\operatorname{Unif}(\bbS^n)}\mleft[\left\|\widehat{\bH}_{b,j}(\bx)-\nabla^2 f_{\sigma}(\bx)\right\|^2\mright]\\
        &\le 
        \mathbb{E}_{\bs_j\sim\operatorname{Unif}(\bbS^n)}\bigl[(n+1)^2 L^2\bigr]=(n+1)^2 L^2.
    \end{aligned}
    $$
    As a result, the matrix Bernstein inequality implies that
    $$
        \mathbb{P}\mleft\{\left\|\sum_{j=1}^b \bigl(\widehat{\bH}_{b,j}(\bx)-\nabla^2 f_{\sigma}(\bx)\bigr)\right\|\ge t\mright\}\le 2 n\exp\mleft(-\frac{t^2/2}{b (n+1)^2 L^2+(n+1) L t/3}\mright)\quad\text{for all $t>0$},
    $$
    from which the desired concentration bound immediately follows. We are done.
\end{proof}

\subsection{Connecting $\nabla^2 f_{\sigma}(\bx)$ with 
$\operatorname{conv}\mleft(\bigcup_{\by\in\bx+\delta\bbB^n}\partial_{\dC}^2 f(\by)\mright)$
}


In this subsection, our goal is to 
establish
a second-order stationarity transfer property from $f_{\sigma}$ to $f$, which parallels~\cite[Theorem~3.1]{lin2022gradient} and will prove useful in later developments.
To this end, 
we begin with some topological properties of $\partial_{\delta} f(\bx)$ and $\partial_{\delta}^2 f(\bx)$.
\begin{lemma}
    Let $L,\delta>0$ and $\bx\in\R^n$ be arbitrary, and $f:\R^n\rightarrow\R$ be $L$-smooth. The sets $\partial_{\delta} f(\bx)$ and $\partial_{\delta}^2 f(\bx)$ are both nonempty, convex, and compact.
\end{lemma}

\begin{proof}
    As $L$-smooth functions are locally Lipschitz by~\cite[Theorem~9.7]{rockafellar2009variational} and~\cite[Theorem~9.2]{rockafellar2009variational}, $\partial_{\delta} f(\bx)$ is compact by~\cite[Proposition~2.3]{goldstein1977optimization}.
    As $\partial_{\dC} f(\bx)\subseteq\partial_{\delta} f(\bx)$ and $\partial_{\dC} f(\bx)\neq\emptyset$ by~\cite[Proposition~4.3.1]{cui2021modern}, we have $\partial_{\delta} f(\bx)\neq\emptyset$. 
    In a similar vein, as $\partial_{\dC}^2 f(\bx)\subseteq\partial_{\delta}^2 f(\bx)$ and $\partial_{\dC}^2 f(\bx)\neq\emptyset$ by~\cite[Proposition~2.6.2(a)]{clarke1990optimization}, it also follows that $\partial_{\delta}^2 f(\bx)\neq\emptyset$.
    It remains to show that $\partial_{\delta}^2 f(\bx)$ is compact.

    Towards that end, we begin by showing that $\bigcup_{\by\in\bx+\delta\bbB^n}\partial_{\dC}^2 f(\by)$ is bounded and closed.
    As $\nabla f$ is $L$-Lipschitz, by~\cite[Proposition~2.6.2(d)]{clarke1990optimization}, the sets $\partial_{\dC}^2 f(\by)$ are uniformly bounded for $\by\in\R^n$; thus, $\bigcup_{\by\in\bx+\delta\bbB^n}\partial_{\dC}^2 f(\by)$ is bounded.
    Besides, as $\bx+\delta\bbB^n$ is compact and $\partial_{\dC}^2 f$ is outer semicontinuous~\cite[Proposition~2.6.2(b)]{clarke1990optimization}, 
    it follows from~\cite[Theorem~5.25(a)]{rockafellar2009variational} that $\bigcup_{\by\in\bx+\delta\bbB^n}\partial_{\dC}^2 f(\by)$ is closed,
    and thus compact. 
    As a result, as the convex hull of any compact set is compact (see, e.g.,~\cite[Corollary~I.2.4]{barvinok2002course}), it follows that $\partial_{\delta}^2 f(\bx)$ is compact, as desired.
\end{proof}



The next result comes without any surprise,
as
$\nabla f_{\sigma}(\bx)$ and $\nabla^2 f_{\sigma}(\bx)$ are the barycenters of 
the pushforward measures $(\nabla f)_{\#}\operatorname{Unif}(\bx+\sigma\bbB^n)$ and $(\nabla^2 f)_{\#}\operatorname{Unif}(\bx+\sigma\bbB^n)$; see Theorem~\ref{thm:smoothing}.

\begin{corollary}\label{cor:connection}
    Let $L,\sigma>0$ and $\bx\in\R^n$ be arbitrary, and $f:\R^n\rightarrow\R$ be $L$-smooth. We have
    $$
        \nabla f_{\sigma}(\bx)\in\partial_{\sigma} f(\bx)\quad\text{and}\quad\nabla^2 f_{\sigma}(\bx)\in\partial_{\sigma}^2 f(\bx).
    $$
\end{corollary}

While the proof is almost verbatim that of~\cite[Theorem~3.1]{lin2022gradient}, we include it here for completeness.

\begin{proof}
    If $\nabla f_{\sigma}(\bx)\notin\partial_{\sigma} f(\bx)$, then~\cite[Corollary~11.4.2]{rock1997convex} implies that there exists a $\by\in\R^n$ for which
    $$
    \begin{aligned}
        \max_{\bz\in\partial_{\sigma} f(\bx)}\by^{\T}\bz<\by^{\T}\nabla f_{\sigma}(\bx)=\by^{\T}\mathbb{E}_{\bu\sim\operatorname{Unif}(\bbB^n)}[\nabla f(\bx+\sigma\bu)]=\mathbb{E}_{\bu\sim\operatorname{Unif}(\bbB^n)}[\by^{\T}\nabla f(\bx+\sigma\bu)].
    \end{aligned}
    $$
    However, as 
    $$
        \nabla f(\bx+\sigma\bu)\in\bigcup_{\by\in\bx+\sigma\bbB^n}\partial_{\dC} f(\by)\subseteq\operatorname{conv}\mleft(\bigcup_{\by\in\bx+\sigma\bbB^n}\partial_{\dC} f(\by)\mright)=\partial_{\sigma} f(\bx)
    $$
    whenever it exists, we have
    $$
        \mathbb{E}_{\bu\sim\operatorname{Unif}(\bbB^n)}[\by^{\T}\nabla f(\bx+\sigma\bu)]\le\mathbb{E}_{\bu\sim\operatorname{Unif}(\bbB^n)}\mleft[\max_{\bz\in\partial_{\sigma} f(\bx)}\by^{\T}\bz\mright]=\max_{\bz\in\partial_{\sigma} f(\bx)}\by^{\T}\bz,
    $$
    a contradiction. In a similar vein, if $\nabla^2 f_{\sigma}(\bx)\notin\partial_{\sigma}^2 f(\bx)$, then there exists a $\bY\in\R^{n\times n}$ such that
    $$
    \begin{aligned}
        \max_{\bZ\in\partial_{\sigma}^2 f(\bx)}\langle\bY,\bZ\rangle<\langle\bY,\nabla^2 f_{\sigma}(\bx)\rangle
        =\mathbb{E}_{\bu\sim\operatorname{Unif}(\bbB^n)}[\langle\bY,\nabla^2 f(\bx+\sigma\bu)\rangle]\le\max_{\bZ\in\partial_{\sigma}^2 f(\bx)}\langle\bY,\bZ\rangle,
    \end{aligned}
    $$
    a contradiction.
\end{proof}

The last technical preparation is
the following lemma, which establishes a fundamental connection between the notion of Goldstein FOSPs and that of (ordinary) FOSPs for $L$-smooth functions. 

\begin{lemma}[{\cite[Proposition~6(ii)]{zhang2020complexity}}]\label{lma:gradient-transfer}
    Let $L,\varepsilon>0$ be arbitrary and $f:\R^n\rightarrow\R$ be $L$-smooth. If $\operatorname{dist}\mleft(\bd{0}, \partial_{\frac{\varepsilon}{3 L}} f(\bx)\mright) \le {\varepsilon}/{3}$, then $\|\nabla f(\bx)\| \le \varepsilon$.
\end{lemma}

We are now in a position to prove 
the desired second-order stationarity transfer property. 

\begin{proposition}\label{prop:Hessian-transfer}
    Let $L,\varepsilon_1,\varepsilon_2,\delta>0$ and $\bx\in\R^n$ be arbitrary, and $f:\R^n\rightarrow\R$ be $L$-smooth. If $\|\nabla f_{\sigma}(\bx)\|\le\varepsilon_g$ and $\lambda_{\min}\bigl(\nabla^2 f_{\sigma}(\bx)\bigr)\ge -\varepsilon_H$, where
    $\varepsilon_g:={\varepsilon_1}/{3}$, $\varepsilon_H:=\varepsilon_2$, and $\sigma:=\min\bigl\{\delta,{\varepsilon_1}/{(3 L)}\bigr\}$,
    then 
    $$
        \|\nabla f(\bx)\| \le \varepsilon_1
        \quad\text{and}\quad
        \min_{\bw\in\bbS^n}\max_{\bA\in\partial_{\delta}^2 f(\bx)}\langle\bA\bw,\bw\rangle\ge-\varepsilon_2.
    $$
\end{proposition}

\begin{proof}
    As $\sigma\le{\varepsilon_1}/{(3 L)}$, Corollary~\ref{cor:connection} implies that $\nabla f_{\sigma}(\bx)\in\partial_{\sigma} f(\bx)\subseteq\partial_{\frac{\varepsilon_1}{3 L}} f(\bx)$, and thus
    $$
        \operatorname{dist}\mleft(\bd{0}, \partial_{\frac{\varepsilon_1}{3 L}} f(\bx)\mright)\le
        \operatorname{dist}\bigl(\bd{0}, \partial_{\sigma} f(\bx)\bigr)\le
        \|\nabla f_{\sigma}(\bx)\|\le\varepsilon_g=\frac{\varepsilon_1}{3};
    $$
    i.e., $\|\nabla f(\bx)\| \le \varepsilon_1$ by Lemma~\ref{lma:gradient-transfer}. On the other hand, as $\sigma\le\delta$, Corollary~\ref{cor:connection} also implies that $\nabla^2 f_{\sigma}(\bx)\in\partial_{\sigma}^2 f(\bx)\subseteq\partial_{\delta}^2 f(\bx)$, and thus
    $$
        \min_{\bw\in\bbS^n}\max_{\bA\in\partial_{\delta}^2 f(\bx)}\langle\bA\bw,\bw\rangle\ge\min_{\bw\in\bbS^n}\max_{\bA\in\partial_{\sigma}^2 f(\bx)}\langle\bA\bw,\bw\rangle
        \ge\min_{\bw\in\bbS^n}\bw^{\T}\nabla^2 f_{\sigma}(\bx)\bw
        \ge -\varepsilon_H=-\varepsilon_2,
    $$
    as desired.
\end{proof}

\section{Computing Goldstein approximate SOSPs of $L$-smooth functions}\label{sec:computing-Goldstein-SOSPs}
In this section, we establish the randomized upper bound by applying the cubic Newton method to $f_{\sigma}$, as described in Algorithm~\ref{alg:random-meta},
which invokes Algorithm~\ref{alg:random-smooth} as a subroutine.

\begin{algorithm}[!h]
\begin{algorithmic}[1]
\REQUIRE An $L$-smooth function $f:\R^n\rightarrow\R$, an initial point $\bx_0\in\R^n$, and tolerances $\varepsilon_1,\varepsilon_2,\delta,\eta>0$. 



\STATE Run Algorithm~\ref{alg:random-smooth} with
$$
    \varepsilon_g:=\frac{\varepsilon_1}{3},\quad
    \varepsilon_H:=\varepsilon_2,\quad\text{and}\quad
    \sigma:=\min\mleft\{\delta,\frac{\varepsilon_1}{3 L}\mright\},
$$
while keeping all the other inputs the same as the ones here;

\end{algorithmic}
\caption{
Computing $(\varepsilon_1,\varepsilon_2,\delta)$-Goldstein approximate SOSPs of
$L$-smooth functions
}
\label{alg:random-meta}
\end{algorithm}

The main result of this section, which follows directly from Proposition~\ref{prop:Hessian-transfer} once Algorithm~\ref{alg:random-smooth} has been shown to be sound with high probability, is as follows.
\begin{theorem}\label{thm:main-Goldstein-SOSP}
    Consider the iterates $\bx_0,\ldots,\bx_{K_{\star}}$ generated by Algorithm~\ref{alg:random-meta}. With probability at least $1-\eta$, 
    there exists some $k_{\star}=0,\ldots,K_{\star}-1$ such that
    $$
        \|\nabla f(\bx_{k_{\star}+1})\| \le \varepsilon_1
        \quad\text{and}\quad
        \min_{\bw\in\bbS^n}\max_{\bA\in\partial_{\delta}^2 f(\bx_{k_{\star}+1})}\langle\bA\bw,\bw\rangle\ge-\varepsilon_2.
    $$
\end{theorem}

In other words, $(\varepsilon_1,\varepsilon_2,\delta)$-Goldstein SOSP is computable in a randomized manner within 
$$
    K_{\star}\lesssim
    \max\mleft\{\frac{\Delta L^{1/2}n^{1/4}}{\sigma^{1/2}\varepsilon_g^{3/2}},\frac{\Delta L^2 n}{\sigma^2 \varepsilon_H^3}\mright\}
    \asymp
    \max\mleft\{\frac{\Delta L^{1/2}n^{1/4}}{\min\{\delta,{\varepsilon_1}/{L}\}^{1/2}\varepsilon_1^{3/2}},\frac{\Delta L^2 n}{\min\{\delta,{\varepsilon_1}/{L}\}^2 \varepsilon_2^3}\mright\}
$$
outer loops of Algorithm~\ref{alg:random-smooth},
which is independent of $\eta$. In particular, when $\varepsilon_1,\varepsilon_2,\delta\asymp\varepsilon$ with $\varepsilon$ small enough and $L\gg 1$,
we have
$$
    K_{\star}
    \lesssim
    {\Delta L^4 n}/{\varepsilon^5},
$$
and the total number of first-order oracle calls is, modulo logarithmic factors, of order
$$
    {\Delta L^8 n^2}/{\varepsilon^9}+{\Delta L^6 n^3}/{\varepsilon^7}.
$$

In closing, we note that the cubic Newton subproblem involved in Algorithm~\ref{alg:random-smooth} is polynomial-time solvable; see, e.g.,~\cite{griewank1981modification},~\cite[Theorem~3.1]{cartis2011adaptive1},~\cite[Section~5.1]{nesterov2006cubic}, and~\cite{zhou2025tight}.
For ease of exposition, we also introduce the following shorthands for the quantities involved in Algorithm~\ref{alg:random-smooth}:
$$
    r_k:=\|\bp_k\|,\quad
    \bq_k:=\widehat{\bg}_{b_g}(\bx_k)-\nabla f_{\sigma}(\bx_k),\quad
    \bQ_k:=\widehat{\bH}_{b_H}(\bx_k)-\nabla^2 f_{\sigma}(\bx_k),\quad
    \alpha_k:=\|\bq_k\|,\quad
    \delta_k:=\|\bQ_k\|.
$$


\subsection{Proof of Theorem~\ref{thm:main-Goldstein-SOSP}: Computing SOSPs of $f_{\sigma}$}
In this subsection, we prove in a discussion manner 
that
Algorithm~\ref{alg:random-smooth} is sound with high probability,
which, as discussed earlier, suffices to imply the main result Theorem~\ref{thm:main-Goldstein-SOSP}.

\begin{theorem}\label{thm:soundedness}
    Consider the iterates $\bx_0,\ldots,\bx_{K_{\star}}$ generated by Algorithm~\ref{alg:random-smooth}. With probability at least $1-\eta$, 
    there exists some $k_{\star}=0,\ldots,K_{\star}-1$ such that
    $$
        \|\nabla f_{\sigma}(\bx_{k_{\star}+1})\|\le\varepsilon_g
        \quad\text{and}\quad
        \lambda_{\min}\bigl(\nabla^2 f_{\sigma}(\bx_{k_{\star}+1})\bigr)\ge -\varepsilon_H.
    $$
\end{theorem}

\begin{algorithm}[!h]
\begin{algorithmic}[1]
\REQUIRE An $L$-smooth function $f:\R^n\rightarrow\R$, an initial point $\bx_0\in\R^n$, and tolerances $\varepsilon_g,\varepsilon_H,\sigma,\eta>0$. 



\STATE Let 
$\xi:={c L \sqrt{n}}/{\sigma}$ and
$$
    K_{\star}:=\left\lceil\frac{24\Delta}{\xi r_{\star}^3}\right\rceil,
$$
where 
$$
    r_{\star}:=\min\mleft\{\sqrt{\frac{16 \varepsilon_g}{17\xi}},\frac{24 \varepsilon_H}{37\xi}\mright\}>0
    \quad\text{and}\quad
    \Delta:=f_{\sigma}(\bx_0)-\min_{\bx\in\R^n}f_{\sigma}(\bx);
$$

\FOR{$k=0,\ldots,K_{\star}-1$}

    \STATE Sample $\bu_1,\ldots,\bu_{b_g}\sim\operatorname{Unif}(\bbB^n)$ independently, and construct
    $$
        \widehat{\bg}_{b_g}(\bx_k):=\frac{1}{b_g}\sum_{j=1}^{b_g} \frac{\nabla f(\bx_k+\sigma\bu_j)+\nabla f(\bx_k-\sigma\bu_j)}{2},
    $$
    where
    $$
        b_g:=\left\lceil\left(\frac{4608 L^2\sigma^2}{\xi^2 r_{\star}^4}+\frac{64 L\sigma}{\xi r_{\star}^2}\right)\log\mleft(\frac{4 K_{\star}}{\eta}\mright)\right\rceil;
    $$

    \STATE Sample $\bs_1,\ldots,\bs_{b_H}\sim\operatorname{Unif}(\bbS^n)$ independently, and construct
    $$
        \widehat{\bH}_{b_H}(\bx_k):=\frac{1}{b_H}\sum_{j=1}^{b_H} \operatorname{Sym}\mleft(\frac{n}{2\sigma}\bigl(\nabla f(\bx_k+\sigma\bs_j)-\nabla f(\bx_k-\sigma\bs_j)\bigr)\bs_j^{\T}\mright),
    $$
    where 
    $$
        b_H:=\left\lceil\left(\frac{1152 (n+1)^2 L^2}{\xi^2 r_{\star}^2}+\frac{16(n+1)L}{\xi r_{\star}}\right)\log\mleft(\frac{4 n K_{\star}}{\eta}\mright)\right\rceil;
    $$

    \STATE Compute $\bp_k\in\R^n$ as an arbitrary global minimizer of
    \begin{equation}\label{eq:cubic-Newton-subproblem}
        \min\mleft\{\widehat{m}_k(\bp):=\bp^{\T}\widehat{\bg}_{b_g}(\bx_k)+\frac{1}{2}\bp^{\T}\widehat{\bH}_{b_H}(\bx_k)\bp+\frac{\xi}{6}\|\bp\|^3:\bp\in\R^n\mright\};
    \end{equation}

    \STATE Update $\bx_{k+1}\leftarrow\bx_k+\bp_k$;

\ENDFOR
\end{algorithmic}
\caption{
Computing an $(\varepsilon_g,\varepsilon_H)$-approximate SOSP of 
$f_{\sigma}$
w.p.\ at least $1-\eta$
}
\label{alg:random-smooth}
\end{algorithm}



To this end, we first recall~\cite[Theorem~3.1]{cartis2011adaptive1} (adapted to our setting), which characterizes the global minimizers of (\ref{eq:cubic-Newton-subproblem}).

\begin{lemma}[{\cite[Theorem~3.1]{cartis2011adaptive1}}]\label{lma:subproblem-optimality}
    A vector $\bp\in\R^n$ is a global minimizer of (\ref{eq:cubic-Newton-subproblem}) if and only if 
    $$
        \bigl(\widehat{\bH}_{b_H}(\bx_k)+\lambda\bI\bigr)\bp=-\widehat{\bg}_{b_g}(\bx_k)
        \quad\text{and}\quad
        \widehat{\bH}_{b_H}(\bx_k)+\lambda\bI\succeq\bO,
        \quad\text{where $\lambda:={\xi\|\bp\|}/{2}$}.
    $$
\end{lemma}

\begin{proof}
    It suffices to note that $\widehat{\bH}_{b_H}(\bx_k)$ is symmetric, and then apply~\cite[Theorem~3.1]{cartis2011adaptive1}.
\end{proof}

We begin the analysis by understanding the model decrease for (\ref{eq:cubic-Newton-subproblem}). While it is also known, we nonetheless include a short proof here to familiarize the readers with the notation.

\begin{lemma}[{\cite[Lemma~4]{nesterov2006cubic}}]\label{lma:mk-pk}
    We have $\widehat{m}_k(\bp_k)\le -\xi r_k^3/12$.
\end{lemma}

\begin{proof}
    As $\bp_k$ is a global minimizer of (\ref{eq:cubic-Newton-subproblem}), it follows from Lemma~\ref{lma:subproblem-optimality} that
    $$
        \bigl(\widehat{\bH}_{b_H}(\bx_k)+\lambda\bI\bigr)\bp_k+\widehat{\bg}_{b_g}(\bx_k)=\bd{0}
        \quad\implies\quad
        \bp_k^{\T}\widehat{\bH}_{b_H}(\bx_k)\bp_k+\lambda r_k^2+\bp_k^{\T}\widehat{\bg}_{b_g}(\bx_k)=0,
    $$
    where $\lambda={\xi r_k}/{2}$;
    i.e.,
    $$
        \bp_k^{\T}\widehat{\bg}_{b_g}(\bx_k)=-\bp_k^{\T}\widehat{\bH}_{b_H}(\bx_k)\bp_k-\frac{\xi}{2} r_k^3.
    $$
    Hence,
    $$
        \widehat{m}_k(\bp_k)=\bp_k^{\T}\widehat{\bg}_{b_g}(\bx_k)+\frac{1}{2}\bp_k^{\T}\widehat{\bH}_{b_H}(\bx_k)\bp_k+\frac{\xi}{6}r_k^3=-\frac{1}{2}\bp_k^{\T}\widehat{\bH}_{b_H}(\bx_k)\bp_k-\frac{\xi}{3} r_k^3.
    $$
    However, Lemma~\ref{lma:subproblem-optimality} also implies that 
    $$
        \widehat{\bH}_{b_H}(\bx_k)\succeq-{\xi r_k\bI}/{2}
        \quad\implies\quad
        \bp_k^{\T}\widehat{\bH}_{b_H}(\bx_k)\bp_k\ge-{\xi r_k^3}/{2}.
    $$
    As a result, it follows that
    $$
        \widehat{m}_k(\bp_k)\le \frac{\xi}{4} r_k^3-\frac{\xi}{3} r_k^3=-\frac{\xi}{12} r_k^3,
    $$
    as desired.
\end{proof}

With the model decrease well understood, we next figure out how it transfers to a decrease in $f_{\sigma}$.

\begin{lemma}\label{lma:real-decrease}
    We have $f_{\sigma}(\bx_k+\bp_k) \le f_{\sigma}(\bx_k) -{\xi r_k^3}/{12} + \alpha_k r_k + \delta_k r_k^2/2$.
\end{lemma}

In comparison with~\cite[Lemma~4]{nesterov2006cubic}, we see that the extra term $\alpha_k r_k + \delta_k r_k^2/2$ is exactly the price we need to pay for using inexact gradients and Hessians.

\begin{proof}
    As $\nabla^2 f_{\sigma}$ is $\xi$-Lipschitz (w.r.t.\ the spectral norm), it follows from~\cite[Lemma~1]{nesterov2006cubic} that
    $$
    \begin{aligned}
        f_{\sigma}(\bx_k+\bp_k) &\le f_{\sigma}(\bx_k)+\bp_k^{\T} \nabla f_{\sigma}(\bx_k)+\frac{1}{2} \bp_k^{\T} \nabla^2 f_{\sigma}(\bx_k) \bp_k + \frac{\xi}{6}r_k^3\\
        &= f_{\sigma}(\bx_k)+\bp_k^{\T}\bigl(\widehat{\bg}_{b_g}(\bx_k)-\bq_k\bigr)+\frac{1}{2} \bp_k^{\T}\bigl(\widehat{\bH}_{b_H}(\bx_k)-\bQ_k\bigr)\bp_k + \frac{\xi}{6}r_k^3\\
        &= f_{\sigma}(\bx_k) +\widehat{m}_k(\bp_k) -\bp_k^{\T}\bq_k - \frac{1}{2} \bp_k^{\T}\bQ_k\bp_k\\
        &\le f_{\sigma}(\bx_k) -\frac{\xi r_k^3}{12} + r_k\|\bq_k\| + \frac{1}{2} r_k^2 \|\bQ_k\| ,
    \end{aligned}
    $$
    where the second inequality follows from Lemma~\ref{lma:mk-pk} and the Cauchy--Schwarz inequality.
\end{proof}


With Lemma~\ref{lma:real-decrease} in place, we next turn to bound the gradient and Hessian of $f_{\sigma}$ at $\bx_k+\bp_k$.

\begin{lemma}\label{lma:bounds-xk+1}
    We have $\|\nabla f_{\sigma}(\bx_k+\bp_k)\|\le \alpha_k+\delta_k r_k+\xi r_k^2$ and $\lambda_{\min}\bigl(\nabla^2 f_{\sigma}(\bx_k+\bp_k)\bigr)\ge -\delta_k -{3\xi r_k}/{2}$.
\end{lemma}

\begin{proof}
    On the one hand,
    by Taylor's Theorem (see, e.g.,~\cite[Theorem~2.1]{nocedal2006numerical}), we have
    $$
    \begin{aligned}
        \nabla f_{\sigma}(\bx_k+\bp_k)&=\nabla f_{\sigma}(\bx_k)+\int_0^1 \nabla^2 f_{\sigma}(\bx_k+t \bp_k) \bp_k \dd t\\
        &=\widehat{\bg}_{b_g}(\bx_k)-\bq_k+\bigl(\widehat{\bH}_{b_H}(\bx_k)-\bQ_k\bigr)\bp_k +\int_0^1 \bigl( \nabla^2 f_{\sigma}(\bx_k+t \bp_k)   -\nabla^2 f_{\sigma}(\bx_k)\bigr) \bp_k \dd t\\
        &=-\frac{\xi r_k}{2}\bp_k-\bq_k-\bQ_k\bp_k +\int_0^1 \bigl( \nabla^2 f_{\sigma}(\bx_k+t \bp_k)   -\nabla^2 f_{\sigma}(\bx_k)\bigr) \bp_k \dd t,
    \end{aligned}
    $$
    and thus
    $$
    \begin{aligned}
        \|\nabla f_{\sigma}(\bx_k+\bp_k)\|&\le
        \frac{\xi}{2} r_k^2+\alpha_k+\delta_k r_k +\left\|\int_0^1 \bigl( \nabla^2 f_{\sigma}(\bx_k+t \bp_k)   -\nabla^2 f_{\sigma}(\bx_k)\bigr) \bp_k \dd t\right\|\\
        &\le
        \frac{\xi}{2} r_k^2+\alpha_k+\delta_k r_k + r_k\int_0^1 \bigl\| \nabla^2 f_{\sigma}(\bx_k+t \bp_k)   -\nabla^2 f_{\sigma}(\bx_k)\bigr\|  \dd t\\
        &\le
        \frac{\xi}{2} r_k^2+\alpha_k+\delta_k r_k + \xi r_k^2\int_0^1 t  \dd t=
        \alpha_k+\delta_k r_k+\xi r_k^2. 
    \end{aligned}
    $$
    On the other hand, as $\|\nabla^2 f_{\sigma}(\bx_k+\bp_k)-\nabla^2 f_{\sigma}(\bx_k)\|\le \xi r_k$, Weyl's inequality implies that
    $$
        \lambda_{\min}\bigl(\nabla^2 f_{\sigma}(\bx_k+\bp_k)\bigr)\ge \lambda_{\min}\bigl(\nabla^2 f_{\sigma}(\bx_k)\bigr)- \xi r_k;
    $$
    see, e.g., the second exercise after~\cite[Corollary~6.3.4]{horn2012matrix}. However, it also follows from Lemma~\ref{lma:subproblem-optimality} that $\widehat{\bH}_{b_H}(\bx_k)\succeq -\xi r_k\bI/2$, and thus by Weyl's inequality again,
    $$
        \lambda_{\min}\bigl(\nabla^2 f_{\sigma}(\bx_k)\bigr)\ge \lambda_{\min}\bigl(\widehat{\bH}_{b_H}(\bx_k)\bigr)-\delta_k\ge
         -\frac{\xi}{2} r_k-\delta_k.
    $$
    By putting the above two pieces together, we have
    $$
        \lambda_{\min}\bigl(\nabla^2 f_{\sigma}(\bx_k+\bp_k)\bigr)\ge -\frac{\xi}{2} r_k-\delta_k - \xi r_k = -\delta_k -\frac{3\xi}{2} r_k,
    $$
    as desired.
\end{proof}






We are now at a position to 
show the soundness of
Algorithm~\ref{alg:random-smooth}, 
conditional on the event that the estimators are sufficiently accurate.

\begin{proposition}
    Consider the iterates $\bx_0,\ldots,\bx_{K_{\star}}$ generated by Algorithm~\ref{alg:random-smooth}.
    Conditional on
    \begin{equation}\label{eq:good-event}
        \mathcal{E}:=\bigcap_{k=0,\ldots,K_{\star}-1}\left\{\|\widehat{\bg}_{b_g}(\bx_k)-\nabla f_{\sigma}(\bx_k)\|\le\frac{\xi r_{\star}^2}{48}~\text{and}~\|\widehat{\bH}_{b_H}(\bx_k)-\nabla^2 f_{\sigma}(\bx_k)\|\le\frac{\xi r_{\star}}{24}\right\},
    \end{equation}
    there must exist some $k_{\star}=0,\ldots,K_{\star}-1$ such that
    $$
        \|\nabla f_{\sigma}(\bx_{k_{\star}+1})\|\le\varepsilon_g
        \quad\text{and}\quad
        \lambda_{\min}\bigl(\nabla^2 f_{\sigma}(\bx_{k_{\star}+1})\bigr)\ge -\varepsilon_H.
    $$
\end{proposition}

\begin{proof}
    If $r_k>r_{\star}$ for some $k=0,\ldots,K_{\star}-1$, then Lemma~\ref{lma:real-decrease} implies that
    $$
    \begin{aligned}
        f_{\sigma}(\bx_{k+1}) &\le f_{\sigma}(\bx_k) -\frac{\xi}{12} r_k^3 + \alpha_k r_k + \frac{\delta_k}{2} r_k^2\\
        &\le f_{\sigma}(\bx_k) -\frac{\xi}{12} r_k^3 + \frac{\xi r_{\star}^2}{48} r_k + \frac{\xi r_{\star}}{48} r_k^2
        < f_{\sigma}(\bx_k) - \frac{\xi}{24} r_k^3 < f_{\sigma}(\bx_k) - \frac{\xi r_{\star}^3}{24}; 
    \end{aligned}
    $$
    i.e., the value of $f_{\sigma}$ can be decreased by at least ${\xi r_{\star}^3}/{24}$.
    Thus, there must exist some $k_{\star}=0,\ldots,K_{\star}-1$ such that $r_{k_{\star}}\le r_{\star}$, as otherwise we would have
    $$
        \Delta\ge f_{\sigma}(\bx_0)-f_{\sigma}(\bx_{K_{\star}})=\sum_{k=0,\ldots,K_{\star}-1}\bigl(f_{\sigma}(\bx_k)-f_{\sigma}(\bx_{k+1})\bigr)> K_{\star} \frac{\xi r_{\star}^3}{24} \ge \frac{24\Delta}{\xi r_{\star}^3}\cdot \frac{\xi r_{\star}^3}{24} = \Delta, 
    $$
    a contradiction. It then immediately follows from Lemma~\ref{lma:bounds-xk+1} that
    $$
        \|\nabla f_{\sigma}(\bx_{k_{\star}+1})\|\le \alpha_{k_{\star}}+\delta_{k_{\star}} r_{k_{\star}}+\xi r_{k_{\star}}^2\le \frac{\xi r_{\star}^2}{48}+\frac{\xi r_{\star}}{24}r_{k_{\star}}+\xi r_{k_{\star}}^2\le 
        \frac{17\xi r_{\star}^2}{16}\le \varepsilon_g,
    $$
    $$
        \lambda_{\min}\bigl(\nabla^2 f_{\sigma}(\bx_{k_{\star}+1})\bigr)\ge -\delta_{k_{\star}} -\frac{3\xi}{2} r_{k_{\star}}\ge 
        -\frac{37\xi r_{\star}}{24}\ge -\varepsilon_H,
    $$
    as desired.
\end{proof}

To show Theorem~\ref{thm:soundedness},
it only remains to show that $\mathcal{E}$ (\ref{eq:good-event}) happens with high probability. 

\begin{proposition}
    We have $\mathbb{P}(\mathcal{E})\ge 1-\eta$.
\end{proposition}

\begin{proof}
    Let $\F_{k-1}$ be the filtration generated by all randomness up to iteration $k-1$ (inclusive). By Proposition~\ref{prop:gradient-estimator} and Proposition~\ref{prop:Hessian-estimator}, for all $k=0,\ldots,K_{\star}-1$, we have
    $$
    \begin{aligned}
        \mathbb{P}\mleft\{\|\widehat{\bg}_{b_g}(\bx_k)-\nabla f_{\sigma}(\bx_k)\|>\frac{\xi r_{\star}^2}{48}\middle|\F_{k-1}\mright\}
        &\le  2\exp\mleft(-\frac{b_g \xi^2 r_{\star}^4/2304}{2 L^2\sigma^2+ L\sigma \xi r_{\star}^2/36}\mright)
        \le   2\exp\mleft(\log\mleft(\frac{\eta}{4 K_{\star}}\mright)\mright)=\frac{\eta}{2 K_{\star}},
    \end{aligned}
    $$
    $$
    \begin{aligned}
        \mathbb{P}\mleft\{\|\widehat{\bH}_{b_H}(\bx_k)-\nabla^2 f_{\sigma}(\bx_k)\|>\frac{\xi r_{\star}}{24}\middle|\F_{k-1}\mright\}
        &\le 2 n\exp\mleft(-\frac{b_H \xi^2 r_{\star}^2/1152}{(n+1)^2 L^2+(n+1) L \xi r_{\star}/72}\mright)\\
        &\le 2 n\exp\mleft(\log\mleft(\frac{\eta}{4 n K_{\star}}\mright)\mright)=\frac{\eta}{2 K_{\star}},
    \end{aligned}
    $$
    and thus by the tower property of conditional probability, 
    $$
        \mathbb{P}\mleft\{\|\widehat{\bg}_{b_g}(\bx_k)-\nabla f_{\sigma}(\bx_k)\|>\frac{\xi r_{\star}^2}{48}
        \mright\}\le \frac{\eta}{2 K_{\star}}
        \quad
        \text{and}\quad
        \mathbb{P}\mleft\{\|\widehat{\bH}_{b_H}(\bx_k)-\nabla^2 f_{\sigma}(\bx_k)\|>\frac{\xi r_{\star}}{24}
        \mright\}\le \frac{\eta}{2 K_{\star}}.
    $$
    It then follows by applying a union bound that
    $$
        \mathbb{P}(\mathcal{E}^{\complement})
        \le\sum_{k=0}^{K_{\star}-1}\mathbb{P}\mleft\{\|\widehat{\bg}_{b_g}(\bx_k)-\nabla f_{\sigma}(\bx_k)\|>\frac{\xi r_{\star}^2}{48}\mright\}+\sum_{k=0}^{K_{\star}-1}\mathbb{P}\mleft\{\|\widehat{\bH}_{b_H}(\bx_k)-\nabla^2 f_{\sigma}(\bx_k)\|>\frac{\xi r_{\star}}{24}\mright\}
        \le
        \eta;
    $$
    i.e., $\mathbb{P}(\mathcal{E})\ge 1-\eta$, as desired.
\end{proof}


\section{Extensions and applications}\label{sec:extensions-applications}
\subsection{Weakly convex functions}
The goal of this subsection is to extend the framework developed earlier to weakly convex functions, 
which
arise widely in machine learning and beyond; e.g., sparse dictionary learning~\cite[Example~2.4]{davis2019stochastic} and robust low-rank matrix recovery~\cite{li2020nonconvex}. (For further examples, we refer interested readers to~\cite[Section~2.1]{davis2019stochastic},~\cite[Section~IV]{li2020understanding}, and the references therein.)

\subsubsection{Ill-posedness of the notion of Clarke SOSPs}
As weakly convex functions are a.e.\ twice differentiable by Alexandrov's theorem (see, e.g.,~\cite[Theorem~13.51]{rockafellar2009variational}), the Clarke generalized Hessian remains well defined. 
However, as it turns out, the construction no longer gives rise to optimality conditions for this extended class of functions.

\begin{example}
    Consider the function $f:\R\rightarrow\R$ given by $f(x):=|x|-x^2$, which is $2$-weakly convex. It is easy to see that $0$ is a (strict) local minimizer of $f$. However, we have
    $$
        \partial_{\dC}^2 f(0)=\operatorname{conv}\mleft(\limsup_{y\rightarrow 0,\,y\neq 0}\{f^{\prime\prime}(y)\}\mright)=\operatorname{conv}\mleft(\limsup_{y\rightarrow 0,\,y\neq 0}\{-2\}\mright)=\{-2\},
    $$
    which is negative definite.
\end{example}

In fact, without further assumptions, the Clarke generalized Hessian could even be empty.

\begin{example}
    Consider the function $f:\R\rightarrow\R$ given by $f(x):=|x|^{4/3}-x^2$, which is $2$-weakly convex. We have
    $$
        \partial_{\dC}^2 f(0)=\operatorname{conv}\mleft(\limsup_{y\rightarrow 0,\,y\neq 0}\{f^{\prime\prime}(y)\}\mright)=\operatorname{conv}\mleft(\limsup_{y\rightarrow 0,\,y\neq 0}\mleft\{\frac{4}{9}|y|^{-2/3}-2\mright\}\mright)=\emptyset.
    $$
\end{example}

\subsubsection{Stationarity/optimality coincidences with Moreau envelopes}
That said, it is well known that a weakly convex function is closely connected to its Moreau envelope, which is 
$L$-smooth;
see, e.g.,~\cite[Lemma~2.5]{davis2022proximal}.
In particular, their FOSPs and local/global minimizers all coincide; see, e.g.,~\cite[Lemma~2.5(4)]{davis2022proximal},~\cite[Theorem~3.5]{khanh2025local}, and~\cite[Corollary~3.4(d)]{hoheisel2020regularization}, respectively. 
In fact, perhaps surprisingly, it turns out that such a pair also enjoys a second-order stationarity coincidence in the sense of
Rockafellar SOSPs.

\begin{corollary}\label{cor:dd-envelope-equivalence}
    Let $\rho>0$ and $\lambda\in(0,1/\rho)$ be arbitrary, and $f:\R^n\rightarrow\R$ be 
    $\rho$-weakly convex. 
    For every $\bx\in\R^n$, the following two statements are equivalent.
    \begin{itemize}
        \item We have $\bd{0}\in\partial f(\bx)$ and $\dd^2 f(\bx;\bd{0})(\bw)\ge 0$ for all $\bw\in\R^n$.

        \item We have $\nabla f_{\lambda}(\bx)=\bd{0}$ and $\dd^2 f_{\lambda}(\bx;\bd{0})(\bw)\ge 0$ for all $\bw\in\R^n$.
    \end{itemize}
\end{corollary}

We formulate Corollary~\ref{cor:dd-envelope-equivalence} as a corollary as it follows immediately from the following proposition, which identifies the second subderivatives of Moreau envelopes with the Moreau envelopes of second subderivatives at FOSPs of weakly convex functions, offering insights complementary to, e.g.,~\cite[Exercise~13.45]{rockafellar2009variational} (with ``for some $\lambda > 0$ sufficiently small'' sharpened) and~\cite[Proposition~4.1]{hang2022augmented} (from convexity to weak convexity, and without twice epi-differentiability).
As an aside, this result also suggests new approaches to analyzing the landscape of nonsmooth functions.

\begin{proposition}\label{prop:dd-of-envelope}
    Let $\rho>0$ and $\lambda\in(0,1/\rho)$ be arbitrary, and $f:\R^n\rightarrow\R$ be 
    $\rho$-weakly convex. 
    If $\bx\in\R^n$ is 
    an FOSP
    of $f$ (i.e., $\bd{0}\in\partial f(\bx)$),
    then
    $$
        \frac{1}{2}\dd^2 f_{\lambda}(\bx;\bd{0})(\bw)=\left(\frac{1}{2}\dd^2 f(\bx;\bd{0})\right)_{\lambda}(\bw)
        \quad\text{for all $\bw\in\R^n$}.
    $$
\end{proposition}

\begin{proof}
    Upon relocating 
    the graph of $f$
    to the origin, we may assume without loss of generality that $\bx=\bd{0}$ and $f(\bx)=0$.
    We begin with the ``$\le$'' direction. 
    As $\bd{0}\in\partial f(\bx)$, it follows from~\cite[Lemma~2.5(4)]{davis2022proximal} that $\operatorname{Prox}_{\lambda f}(\bx)=\bx$, and thus
    $$
        f_{\lambda}(\bx)=f\bigl(\operatorname{Prox}_{\lambda f}(\bx)\bigr)+\frac{1}{2\lambda}\bigl\|\operatorname{Prox}_{\lambda f}(\bx)-\bx\bigr\|^2=f(\bx)+\frac{1}{2\lambda}\|\bx-\bx\|^2=f(\bx)=0.
    $$
    As a result, for all $t>0$ and $\bw^{\prime}\in\R^n$, we have
    \begin{equation}\label{eq:Delta-as-envelope}
    \begin{aligned}
        \frac{f_{\lambda}(\bx+t \bw^{\prime})-f_{\lambda}(\bx)-t \bd{0}^{\T} \bw^{\prime}}{t^2/2}&=
        \frac{f_{\lambda}(t \bw^{\prime})}{t^2/2}\\
        &=
        \frac{2}{t^2}\min\mleft\{f(\by)+\frac{1}{2\lambda}\|\by-t \bw^{\prime}\|^2:\by\in\R^n\mright\}\\
        &=
        \frac{2}{t^2}\min\mleft\{f(t \by)+\frac{t^2}{2\lambda}\|\by-\bw^{\prime}\|^2:\by\in\R^n\mright\}
        \\
        &=
        \min\mleft\{\frac{f(t \by)}{t^2/2}+\frac{1}{\lambda}\|\by-\bw^{\prime}\|^2:\by\in\R^n\mright\}.
    \end{aligned}
    \end{equation}
    Now, for all $\bu\in\R^n$, 
    as
    $$
        \dd^2 f(\bx;\bd{0})(\bu)=\liminf_{t\searrow 0,\,\bu^{\prime}\rightarrow\bu}\frac{f(\bx+t \bu^{\prime})-f(\bx)-t \bd{0}^{\T} \bu^{\prime}}{t^2/2},
    $$
    it follows from~\cite[Lemma~1.7]{rockafellar2009variational} that there must exist a sequence $(t_1,\bu_1),(t_2,\bu_2),\ldots$ with $t_k\searrow 0$ and $\bu_k\rightarrow\bu$ as $k\rightarrow\infty$ such that
    $$
        \dd^2 f(\bx;\bd{0})(\bu)=\lim_{k\rightarrow\infty}\frac{f(\bx+t_k \bu_k)-f(\bx)-t_k \bd{0}^{\T} \bu_k}{t_k^2/2}=
        \lim_{k\rightarrow\infty}\frac{f(t_k \bu_k)}{t_k^2/2}.
    $$
    On the other hand, for all $\bw\in\R^n$, by letting $t=t_k$ and $\bw^{\prime}=\bw$ in (\ref{eq:Delta-as-envelope}) and taking $\bu_k$ as a feasible solution to (\ref{eq:Delta-as-envelope}), 
    it also follows that
    $$
    \begin{aligned}
        \frac{f_{\lambda}(\bx+t_k \bw)-f_{\lambda}(\bx)-t_k \bd{0}^{\T} \bw}{t_k^2/2}&=\min\mleft\{\frac{f(t_k \by)}{t_k^2/2}+\frac{1}{\lambda}\|\by-\bw\|^2:\by\in\R^n\mright\}\le
        \frac{f(t_k \bu_k)}{t_k^2/2}+\frac{1}{\lambda}\|\bu_k-\bw\|^2.
    \end{aligned}
    $$
    By putting the above two pieces together, we have
    $$
    \begin{aligned}
        \liminf_{t\searrow 0,\,\bw^{\prime}\rightarrow\bw}\frac{f_{\lambda}(\bx+t \bw^{\prime})-f_{\lambda}(\bx)-t \bd{0}^{\T} \bw^{\prime}}{t^2/2}&\le
        \liminf_{k\rightarrow\infty}\frac{f_{\lambda}(\bx+t_k \bw)-f_{\lambda}(\bx)-t_k \bd{0}^{\T} \bw}{t_k^2/2}\\
        &\le\liminf_{k\rightarrow\infty}\left(\frac{f(t_k \bu_k)}{t_k^2/2}+\frac{1}{\lambda}\|\bu_k-\bw\|^2\right)\\
        &=\lim_{k\rightarrow\infty}\frac{f(t_k \bu_k)}{t_k^2/2}+\lim_{k\rightarrow\infty}\frac{1}{\lambda}\|\bu_k-\bw\|^2\\
        &=\dd^2 f(\bx;\bd{0})(\bu)+\frac{1}{\lambda}\|\bu-\bw\|^2;
    \end{aligned}
    $$
    i.e.,
    $$
        \frac{1}{2}\dd^2 f_{\lambda}(\bx;\bd{0})(\bw)\le\frac{1}{2}\dd^2 f(\bx;\bd{0})(\bu)+\frac{1}{2\lambda}\|\bu-\bw\|^2
        \quad\text{for all $\bu,\bw\in\R^n$},
    $$
    or in other words, for all $\bw\in\R^n$,
    $$
        \frac{1}{2}\dd^2 f_{\lambda}(\bx;\bd{0})(\bw)\le\min\mleft\{\frac{1}{2}\dd^2 f(\bx;\bd{0})(\bu)+\frac{1}{2\lambda}\|\bu-\bw\|^2:\bu\in\R^n\mright\}=\left(\frac{1}{2}\dd^2 f(\bx;\bd{0})\right)_{\lambda}(\bw),
    $$
    as desired.

    To see the other direction, we first observe through~\cite[Lemma~1.7]{rockafellar2009variational} again that for all $\bw\in\R^n$ there must exist a sequence $(t_1,\bw_1),(t_2,\bw_2),\ldots$ with $t_k\searrow 0$ and $\bw_k\rightarrow\bw$ as $k\rightarrow\infty$ such that
    $$
    \begin{aligned}
        \dd^2 f_{\lambda}(\bx;\bd{0})(\bw)&=\lim_{k\rightarrow\infty}\frac{f_{\lambda}(\bx+t_k \bw_k)-f_{\lambda}(\bx)-t_k \bd{0}^{\T} \bw_k}{t_k^2/2}\\
        &=\lim_{k\rightarrow\infty}
        \min\mleft\{\frac{f(t_k \by)}{t_k^2/2}+\frac{1}{\lambda}\|\by-\bw_k\|^2:\by\in\R^n\mright\},
    \end{aligned}
    $$
    where the last equality follows from (\ref{eq:Delta-as-envelope}) with $t=t_k$ and $\bw^{\prime}=\bw_k$. Now, let $\by_1,\by_2,\ldots$ be such that
    $$
        \frac{f(t_k \by_k)}{t_k^2/2}+\frac{1}{\lambda}\|\by_k-\bw_k\|^2-\frac{1}{k}\le\min\mleft\{\frac{f(t_k \by)}{t_k^2/2}+\frac{1}{\lambda}\|\by-\bw_k\|^2:\by\in\R^n\mright\}
        \quad\text{for all $k=1,2,\ldots$},
    $$
    and thus
    $$
        \limsup_{k\rightarrow\infty}\left(\frac{f(t_k \by_k)}{t_k^2/2}+\frac{1}{\lambda}\|\by_k-\bw_k\|^2\right)\le
        \dd^2 f_{\lambda}(\bx;\bd{0})(\bw).
    $$
    We claim that $\limsup_{k\rightarrow\infty}\|\by_k\|<\infty$. To this end, we first observe that
    $$
    \begin{aligned}
        \dd^2 f_{\lambda}(\bx;\bd{0})(\bw)&=\lim_{k\rightarrow\infty}
        \min\mleft\{\frac{f(t_k \by)}{t_k^2/2}+\frac{1}{\lambda}\|\by-\bw_k\|^2:\by\in\R^n\mright\}
        \le\limsup_{k\rightarrow\infty}
        \frac{1}{\lambda}\|\bw_k\|^2<\infty,
    \end{aligned}
    $$
    as $\bw_1,\bw_2,\ldots$ is a convergent sequence. 
    Then, we estimate
    $$
    \begin{aligned}
        \frac{f(t_k \by_k)}{t_k^2/2}+\frac{1}{\lambda}\|\by_k-\bw_k\|^2
        \ge -\rho\|\by_k\|^2 +\frac{1}{\lambda}\|\by_k-\bw_k\|^2=
        \left(\frac{1}{\lambda}-\rho\right)\|\by_k\|^2-\frac{2}{\lambda}\by_k^{\T}\bw_k+\frac{1}{\lambda}\|\bw_k\|^2,
    \end{aligned}
    $$
    where the inequality follows from~\cite[Lemma~2.1(3)]{davis2019stochastic}. However, as $\bw_1,\bw_2,\ldots$ is convergent, it must be uniformly bounded; i.e., $\|\bw_k\|\le M$ for some $M>0$ and all $k=1,2,\ldots$. This implies that 
    $$
        \left(\frac{1}{\lambda}-\rho\right)\|\by_k\|^2-\frac{2}{\lambda}\by_k^{\T}\bw_k+\frac{1}{\lambda}\|\bw_k\|^2\ge
        \left(\frac{1}{\lambda}-\rho\right)\|\by_k\|^2-\frac{2 M}{\lambda}\|\by_k\|
        \rightarrow\infty\quad\text{as $\|\by_k\|\rightarrow\infty$},
    $$
    as ${1}/{\lambda}-\rho>0$ as assumed. As a result, if $\limsup_{k\rightarrow\infty}\|\by_k\|=\infty$, then
    $$
        \dd^2 f_{\lambda}(\bx;\bd{0})(\bw)\ge\limsup_{k\rightarrow\infty}\left(\frac{f(t_k \by_k)}{t_k^2/2}+\frac{1}{\lambda}\|\by_k-\bw_k\|^2\right)\ge
        \limsup_{k\rightarrow\infty}\left(\left(\frac{1}{\lambda}-\rho\right)\|\by_k\|^2-\frac{2 M}{\lambda}\|\by_k\|\right)=\infty,
    $$
    a contradiction. It then follows from the claim that $\by_1,\by_2,\ldots$ is uniformly bounded, and thus admits a subsequence $\by_{k_1},\by_{k_2},\ldots$ with $\lim_{\ell\rightarrow\infty}\by_{k_{\ell}}=\by_0$ for some $\by_0\in\R^n$ by the Bolzano--Weierstrass theorem, which further implies that
    $$
    \begin{aligned}
        \dd^2 f_{\lambda}(\bx;\bd{0})(\bw)&\ge\limsup_{k\rightarrow\infty}\left(\frac{f(t_k \by_k)}{t_k^2/2}+\frac{1}{\lambda}\|\by_k-\bw_k\|^2\right)\\
        &\ge\limsup_{\ell\rightarrow\infty}\left(\frac{f(t_{k_{\ell}} \by_{k_{\ell}})}{t_{k_{\ell}}^2/2}+\frac{1}{\lambda}\|\by_{k_{\ell}}-\bw_{k_{\ell}}\|^2\right)\\
        &\ge\liminf_{\ell\rightarrow\infty}\frac{f(t_{k_{\ell}} \by_{k_{\ell}})}{t_{k_{\ell}}^2/2}+\frac{1}{\lambda}\|\by_0-\bw\|^2\\
        &\ge\liminf_{t\searrow 0,\,\by^{\prime}\rightarrow\by_0}\frac{f(t \by^{\prime})}{t^2/2}+\frac{1}{\lambda}\|\by_0-\bw\|^2\\
        &=
        \dd^2 f(\bx;\bd{0})(\by_0)+\frac{1}{\lambda}\|\by_0-\bw\|^2\\
        &\ge \min\mleft\{\dd^2 f(\bx;\bd{0})(\bu)+\frac{1}{\lambda}\|\bu-\bw\|^2:\bu\in\R^n\mright\}\\
        &=2\left(\frac{1}{2}\dd^2 f(\bx;\bd{0})\right)_{\lambda}(\bw);
    \end{aligned}
    $$
    i.e., 
    $$
        \frac{1}{2}\dd^2 f_{\lambda}(\bx;\bd{0})(\bw)\ge \left(\frac{1}{2}\dd^2 f(\bx;\bd{0})\right)_{\lambda}(\bw)
        \quad\text{for all $\bw\in\R^n$},
    $$
    as desired.
\end{proof}



We are now ready to prove Corollary~\ref{cor:dd-envelope-equivalence}.

\begin{myproof}{Corollary~\ref{cor:dd-envelope-equivalence}}
    If $\bd{0}\in\partial f(\bx)$ and $\dd^2 f(\bx;\bd{0})(\bw)\ge 0$ for all $\bw\in\R^n$, then $\nabla f_{\lambda}(\bx)=\bd{0}$ by~\cite[Lemma~2.5(4)]{davis2022proximal} and 
    $$
        \frac{1}{2}\dd^2 f_{\lambda}(\bx;\bd{0})(\bw)=\left(\frac{1}{2}\dd^2 f(\bx;\bd{0})\right)_{\lambda}(\bw)=\min\mleft\{\frac{1}{2}\dd^2 f(\bx;\bd{0})(\bu)+\frac{1}{2\lambda}\|\bu-\bw\|^2:\bu\in\R^n\mright\}\ge 0.
    $$
    If $\nabla f_{\lambda}(\bx)=\bd{0}$ and $\dd^2 f_{\lambda}(\bx;\bd{0})(\bw)\ge 0$ for all $\bw\in\R^n$, then $\bd{0}\in\partial f(\bx)$ by~\cite[Lemma~2.5(4)]{davis2022proximal} and
    $$
        \frac{1}{2}\dd^2 f(\bx;\bd{0})(\bw)\ge\left(\frac{1}{2}\dd^2 f(\bx;\bd{0})\right)_{\lambda}(\bw)=\frac{1}{2}\dd^2 f_{\lambda}(\bx;\bd{0})(\bw)\ge 0,
    $$
    as desired.
\end{myproof}




While Corollary~\ref{cor:dd-envelope-equivalence} already reduces the problem of computing a Rockafellar SOSP of a potentially nonsmooth function to a simpler setting (i.e., $L$-smooth functions) without any loss, it is still natural to wonder whether the reduced problem can in fact be even easier; i.e., whether the class of Moreau envelopes generated by all weakly convex functions can be a proper subset of all $L$-smooth functions. The following realizability result of Moreau envelopes, which follows as an immediate consequence of the proximal-hull techniques in~\cite{rockafellar2009variational},
offers a refutation of such a possibility through a stronger statement:
Every $L$-smooth function can be generated as the Moreau envelope of another $L^{\prime}$-smooth function, rather than a weakly convex function.

\begin{proposition}
    Let $L>0$ and $0<\lambda<1/L$ be arbitrary, and $g:\R^n\rightarrow\R$ be $L$-smooth. We have 
    $$
        g=\bigl(-(-g)_{\lambda}\bigr)_{\lambda},
        \quad\text{where $-(-g)_{\lambda}$ is $\max\mleft\{\frac{1}{\lambda},\frac{L}{1-\lambda L}\mright\}$-smooth}.
    $$
\end{proposition}

\begin{proof}
    As $g$ is $L$-smooth, $-g$ is $L$-smooth, and thus $L$-weakly convex; see, e.g.,~\cite[Proposition~2.3]{pougkakiotis2023zeroth} and~\cite[Proposition~4.12]{vial1983strong}. Hence, the function
    $$
        -g+\frac{1}{2\lambda}\|\bullet\|^2=\left(-g+\frac{L}{2}\|\bullet\|^2\right)+\left(\frac{1}{2\lambda}-\frac{L}{2}\right)\|\bullet\|^2
    $$
    is convex. It then follows from~\cite[Example~11.26(d)]{rockafellar2009variational} that $-g$ is $\lambda$-proximal; i.e., 
    $$
        -g=-\bigl(-(-g)_{\lambda}\bigr)_{\lambda}
        \quad\iff\quad
        g=\bigl(-(-g)_{\lambda}\bigr)_{\lambda}
    $$
    by~\cite[Example~1.44]{rockafellar2009variational}. It remains to note that as $-g$ is $L$-weakly convex, $(-g)_{\lambda}$ is $\max\{1/\lambda,L/(1-\lambda L)\}$-smooth by~\cite[Lemma~2.5(3)]{davis2022proximal}, and $-(-g)_{\lambda}$ is thus $\max\{1/\lambda,L/(1-\lambda L)\}$-smooth as well. 
\end{proof}

\subsubsection{Clarke/Goldstein SOSPs via the Moreau envelope}

Drawing on the connections discussed earlier, we propose to define the notion of Clarke SOSPs for weakly convex functions in terms of their Moreau envelopes, which is justified by the following proposition, and also addresses in part the second open question raised in~\cite[Section~5]{khanh2025local}.

\begin{proposition}\label{prop:optimality-via-Moreau}
    Let $\rho>0$ and $\lambda\in(0,1/\rho)$ be arbitrary, $f:\R^n\rightarrow\R$ be
    $\rho$-weakly convex, and $\bx\in\R^n$ be a local minimizer of $f$. 
    We have
    $$
        \nabla f_{\lambda}(\bx)=\bd{0}
        \quad\text{and}\quad
        \min_{\bw\in\bbS^n}\max_{\bz\in\partial_{\dC}^2 f_{\lambda}(\bx)(\bw)}\langle\bz,\bw\rangle\ge 0.
    $$
\end{proposition}

\begin{proof}
    We shall apply~\cite[Theorem~3.5]{khanh2025local} to show that $\bx$ is also a local minimizer of $f_{\lambda}$. To this end, we first claim that a prox-bounded threshold of $f$ (in the sense of~\cite[Lemma~3.1(iii)]{khanh2025local}) can be $1/\rho$; i.e., $f_{\mu}(\bx)>-\infty$ for all $\mu\in(0,1/\rho)$ and $\bx\in\R^n$. But this is trivial, as $\by\mapsto f(\by)+\frac{1}{2\mu}\|\by-\bx\|^2$ is strongly convex (see, e.g.,~\cite[Remark~2.5(iii)]{khanh2025inexact} for more details), and thus lower bounded.
    Hence, by~\cite[Theorem~3.5]{khanh2025local}, we have $\bx$ is a local minimizer of $f_{\lambda}$. As $f_{\lambda}$ is $\max\{\lambda^{-1}, {\rho}/{(1-\lambda \rho)}\}$-smooth by~\cite[Lemma~2.5(3)]{davis2022proximal}, the desired claims then immediately follow by applying Fermat’s rule (see, e.g.,~\cite[Theorem~10.1]{rockafellar2009variational}) and the second-order optimality condition~\cite[Theorem~3.1]{hiriart1984generalized}.
\end{proof}

With Proposition~\ref{prop:optimality-via-Moreau} in place, we are now allowed to define Goldstein approximate SOSPs for weakly convex functions.

\begin{definition}
\label{def:Goldstein-SOSP-weakly-convex}
    Given constants $\varepsilon_1,\varepsilon_2,\delta\ge 0$, a $\rho$-weakly convex function $f:\R^n\rightarrow\R$ for some $\rho>0$, 
    and a point $\bx\in\R^n$, we say that $\bx$ is an $(\varepsilon_1,\varepsilon_2,\delta)$-Goldstein approximate SOSP of $f$ if 
    $$
        \|\nabla f_{{1}/{(2\rho)}}(\bx)\|\le\varepsilon_1
        \quad\text{and}\quad
        \min_{\bw\in\bbS^n}\max_{\bA\in\partial_{\delta}^2 f_{{1}/{(2\rho)}}(\bx)}\langle\bA\bw,\bw\rangle\ge-\varepsilon_2.
    $$
\end{definition}

\subsubsection{Computing Goldstein approximate SOSPs}\label{sec:weakly-convex-Goldstein-computation}
We now apply Algorithm~\ref{alg:random-meta} to compute Goldstein approximate SOSPs of weakly convex functions.

\begin{algorithm}[!h]
\begin{algorithmic}[1]
\REQUIRE A $\rho$-weakly convex function $g:\R^n\rightarrow\R$, an initial point $\bx_0\in\R^n$, tolerances $\varepsilon_1,\varepsilon_2,\delta,\eta>0$. 



\STATE Run Algorithm~\ref{alg:random-meta} with
$$
    f:=g_{{1}/{(2\rho)}}
    \quad\text{and}\quad
    L:=2\rho,
$$
while keeping all the other inputs the same as the ones here;

\end{algorithmic}
\caption{
Computing $(\varepsilon_1,\varepsilon_2,\delta)$-Goldstein approximate SOSPs of
weakly convex functions
}
\label{alg:random-wc-meta}
\end{algorithm}

We remark that the only additional ingredient needed in the implementation is the evaluation of $\nabla g_{{1}/{(2\rho)}}$. But as $\nabla g_{{1}/{(2\rho)}}(\bx)=2\rho\bigl(\bx-\operatorname{Prox}_{{g}/{(2\rho)}}(\bx)\bigr)$ by~\cite[Lemma~2.5(1)]{davis2022proximal}, it suffices to compute
$$
    \operatorname{Prox}_{{g}/{(2\rho)}}(\bx)=\arg\min\bigl\{g(\by)+\rho\|\by-\bx\|^2:\by\in\R^n\bigr\},
$$
which is strongly convex and can thus be solved efficiently; see, e.g.,~\cite[Section~3.4]{bubeck2015convex}.

\begin{corollary}\label{cor:weakly-convex-Goldstein-SOSP}
    Consider the iterates $\bx_0,\ldots,\bx_{K_{\star}}$ generated by Algorithm~\ref{alg:random-wc-meta}. With probability at least $1-\eta$, 
    there exists some $k_{\star}=0,\ldots,K_{\star}-1$ such that
    $$
        \|\nabla g_{{1}/{(2\rho)}}(\bx_{k_{\star}+1})\| \le \varepsilon_1
        \quad\text{and}\quad
        \min_{\bw\in\bbS^n}\max_{\bA\in\partial_{\delta}^2 g_{{1}/{(2\rho)}}(\bx_{k_{\star}+1})}\langle\bA\bw,\bw\rangle\ge-\varepsilon_2.
    $$
\end{corollary}

\begin{proof}
    It suffices to note that $g_{{1}/{(2\rho)}}$ is $2\rho$-smooth by~\cite[Lemma~2.5(3)]{davis2022proximal}, and apply Theorem~\ref{thm:main-Goldstein-SOSP}.
\end{proof}

\subsection{Second-order bilevel optimization and beyond}
In this subsection, we focus on the following optimistic/pessimistic bilevel optimization problems
$$
    \min_{\bx\in\bbX}\min_{\by\in\bbY^{\star}(\bx)}u(\bx,\by)
    \quad\text{and}\quad
    \min_{\bx\in\bbX}\max_{\by\in\bbY^{\star}(\bx)}u(\bx,\by),\quad\text{where $\bbY^{\star}(\bx):=\operatorname{Arg}\min\{\ell(\bx,\by):\by\in\bbY\}$}.
$$
where $u,\ell:\bbX\times\bbY\rightarrow\R$, $\bbX\subseteq\R^n$, and $\bbY\subseteq\R^m$. Equivalently, one may also pose the problems as
$$
    \min\mleft\{h_o(\bx):=\min_{\by\in\bbY^{\star}(\bx)}u(\bx,\by):\bx\in\bbX\mright\}
    \quad\text{and}\quad
    \min\mleft\{h_p(\bx):=\max_{\by\in\bbY^{\star}(\bx)}u(\bx,\by):\bx\in\bbX\mright\},
$$
respectively, where $h_o,h_p:\bbX\rightarrow\R$ are referred to as the hyperobjective functions.
In cases where $h_o=h_p$, we simply write $h$ for both; e.g., when $|\bbY^{\star}(\bx)|=1$ for all $\bx\in\bbX$, and in such a case, we also write $\by^{\star}$ for $\bbY^{\star}$.
In the sequel, a bilevel regime, e.g., NC--SC, means that $h$ is nonconvex (in general) whereas $\ell(\bx,\bullet)$ is strongly convex (for all $\bx\in\bbX$).
When discussing the NC--SC
and NC--P\L{}
regimes, we take $\bbX=\R^n$ and $\bbY=\R^m$, whereas in other regimes $\bbX$ and $\bbY$ may be more general.
As an aside, as bilevel optimization is not the main focus of this paper, we leave issues such as inexactness to interested readers and future work.
In closing, we note that bilevel programs also arise widely in machine learning and beyond; e.g., hyperparameter optimization~\cite{chen2024lower}.
(For further examples, we refer interested readers to~\cite[Section~1]{chen2024finding},~\cite[Section~1]{chen2025set}, and the references therein.)

\subsubsection{The NC--SC regime without third-order differentiability}
\label{sec:NC-SC-bilevel}
Existing studies on second-order NC--SC bilevel optimization are typically based on the third-order differentiability of $\ell$ to ensure that $h$ is twice differentiable, thereby enabling the use of the machinery of second-order smooth optimization; see, e.g.,~\cite[Assumption~3]{yang2023accelerating},~\cite[Assumption~3.2]{chen2025near}, and~\cite[Assumption~3]{huang2025efficiently}.
However, the third-order differentiability condition can be rather demanding, and it is thus very natural to ask whether we can still say something meaningful in its absence. With the techniques developed in this paper, the answer is in the affirmative, as $h$ remains $L$-smooth even without the condition, which follows from, e.g.,~\cite[Lemma~2.2]{ghadimi2018approximation} and~\cite[Proposition~3.1]{chen2025near}, and is quoted below. 

\begin{lemma}[{\cite[Proposition~3.1]{chen2025near}}]\label{lma:bilevel-L-smooth}
    Suppose that $\ell(\bx,\bullet)$ is $\mu$-strongly convex,
    $\ell$ has $L_{\ell}$-Lipschitz gradient and $\rho_{\ell}$-Lipschitz Hessian, 
    $u(\bx,\bullet)$ is $C_u$-Lipschitz,
    $u$ is 
    twice continuously differentiable
    and $L_u$-smooth, 
    and $h$ is lower bounded. 
    Then, we have $h$ is $L_h$-smooth, where 
    $$
        L_h\lesssim \max\{C_u,L_u,L_{\ell},\rho_{\ell}\}^4/\mu^3.
    $$
\end{lemma}

Of course, away from the condition, the twice differentiability of $h$ can no longer be guaranteed.

\begin{example}
    Consider the functions $u,\ell:\R^2\rightarrow\R$ and $r:\R\rightarrow\R$, given by $u(x,y):=x^2+y$, $\ell(x,y):=(y-x)^2/2+r(y)$, and
    $$
        r(y):=
        \begin{dcases}
        \frac{1}{6}|y|^3, & |y| \le 1, \\
        \frac{1}{2} y^2-\frac{1}{2}|y|+\frac{1}{6}, & |y| > 1,
        \end{dcases}
    $$
    respectively. It can be calculated that
    $$
        r^{\prime}(y)=
        \begin{dcases}
        \frac{1}{2}y|y|, & |y| \le 1, \\
        y-\frac{1}{2}\operatorname{sign}(y), & |y| > 1,
        \end{dcases}
        \quad\text{and}\quad
        r^{\prime\prime}(y)=
        \begin{dcases}
        |y|, & |y| \le 1, \\
        1, & |y| > 1,
        \end{dcases}
    $$
    which imply all the assumptions in Lemma~\ref{lma:bilevel-L-smooth} except the lower boundedness of $h$, and that $\ell(x,\bullet)$ is not three times differentiable. To see that $h$ is indeed lower bounded, we calculate
    $$
        y^{\star}(x)
        =
        \begin{dcases}
            (2 x-1)/4, & x\leq -{3}/{2},\\
            1-\sqrt{1-2 x}, & -{3}/{2}<x\leq 0,\\
            \sqrt{2 x+1}-1, & 0<x\leq {3}/{2},\\
            (2 x+1)/4, & x>3/2,
        \end{dcases}
    $$
    and thus
    $$
        \min\{h(x):x\in\R\}
        =\min\bigl\{x^2+y^{\star}(x):x\in\R\bigr\}\approx -0.1823 >-\infty,
    $$
    as desired. However, as $y^{\star}$ is not twice differentiable at $0$, neither is $h$.
\end{example}

We now apply Algorithm~\ref{alg:random-meta} to compute Goldstein approximate SOSPs of NC--SC bilevel programs.

\begin{algorithm}[!h]
\begin{algorithmic}[1]
\REQUIRE Functions $u,\ell:\R^n\times\R^m\rightarrow\R$ satisfying all the assumptions in Lemma~\ref{lma:bilevel-L-smooth}, an initial point $\bx_0\in\R^n$, tolerances $\varepsilon_1,\varepsilon_2,\delta,\eta>0$. 



\STATE Run Algorithm~\ref{alg:random-meta} with
$$
    f:=h=u\bigl(\bullet,\by^{\star}(\bullet)\bigr)
    \quad\text{and}\quad
    L:=L_h,
$$
while keeping all the other inputs the same as the ones here;

\end{algorithmic}
\caption{
Computing $(\varepsilon_1,\varepsilon_2,\delta)$-Goldstein approximate SOSPs of
NC--SC bilevel programs
}
\label{alg:random-NC-SC-meta}
\end{algorithm}

We remark that the only additional ingredient needed in the implementation is the evaluation of $\nabla h$; i.e., the hypergradient. But as $\nabla h$ admits an explicit closed-form expression (see, e.g.,~\cite[(2)]{chen2025near}), it evaluation amounts to nothing more than matrix multiplications, which can be done efficiently.

\begin{corollary}
    Consider the iterates $\bx_0,\ldots,\bx_{K_{\star}}$ generated by Algorithm~\ref{alg:random-NC-SC-meta}. With probability at least $1-\eta$, 
    there exists some $k_{\star}=0,\ldots,K_{\star}-1$ such that
    $$
        \|\nabla h(\bx_{k_{\star}+1})\| \le \varepsilon_1
        \quad\text{and}\quad
        \min_{\bw\in\bbS^n}\max_{\bA\in\partial_{\delta}^2 h(\bx_{k_{\star}+1})}\langle\bA\bw,\bw\rangle\ge-\varepsilon_2.
    $$
\end{corollary}

\begin{proof}
    It suffices to combine Lemma~\ref{lma:bilevel-L-smooth} with Theorem~\ref{thm:main-Goldstein-SOSP}.
\end{proof}

\subsubsection{The NC--P\L{} regimes}
While the NC--SC regime gives rise to an $L$-smooth hyperobjective function, the strong convexity assumption can be demanding. 
In fact, with the help of the Polyak--\L{}ojasiewicz inequality~\cite{polyak1963gradient}, one can move beyond the strong convexity, albeit with a slight reduction in smoothness.
(Let us restrict attention to the pessimistic case as the other case is a mirror image.)

\begin{lemma}[{\cite[Theorem~3]{chen2025set}}]\label{lma:bilevel-weakly-convex}
    Suppose that $\ell$ is $L_{\ell}$-smooth and twice differentiable with $\nabla\nabla_{\by}\ell$ being $H_{\ell}$-Lipschitz, $\bbY^{\star}(\bx)$ is nonempty, closed, and convex for all $\bx\in\R^n$, $u$ is $M_u$-Lipschitz and $L_u$-smooth, $\operatorname{dom}(h_p)\neq\emptyset$, and there exists a constant $\tau>0$ such that
    $$
        \operatorname{dist}\bigl(\by, \bbY^{\star}(\bx)\bigr) \le \tau\|\nabla_{\by} \ell(\bx, \by)\|
        \quad\text{for all $\bx\in\R^n$ and $\by\in\R^m$}.
    $$
    Then, $h_p$ is $\bigl(M_u L_{\bbY^{\star}}+L_u(1+L_{\bbY^{\star}})\bigr)$-weakly convex, where $L_{\bbY^{\star}}:=\max\bigl\{2 H_{\ell} \tau(1+9 L_{\ell}^2 \tau^2), 4 L_{\ell}^2 \tau^2\bigr\}$.
\end{lemma}

We now apply Algorithm~\ref{alg:random-wc-meta} to compute Goldstein approximate SOSPs of NC--P\L{} bilevel programs.

\begin{algorithm}[!h]
\begin{algorithmic}[1]
\REQUIRE Functions $u,\ell:\R^n\times\R^m\rightarrow\R$ satisfying all the assumptions in Lemma~\ref{lma:bilevel-weakly-convex}, an initial point $\bx_0\in\R^n$, tolerances $\varepsilon_1,\varepsilon_2,\delta,\eta>0$. 



\STATE Run Algorithm~\ref{alg:random-wc-meta} with
$$
    g:=h_p=\max_{\by\in\bbY^{\star}(\bullet)}u(\bullet,\by)
    \quad\text{and}\quad
    \rho:=M_u L_{\bbY^{\star}}+L_u(1+L_{\bbY^{\star}}),
$$
while keeping all the other inputs the same as the ones here;

\end{algorithmic}
\caption{
Computing $(\varepsilon_1,\varepsilon_2,\delta)$-Goldstein approximate SOSPs of
NC--P\L{} bilevel programs
}
\label{alg:random-NC-PL-meta}
\end{algorithm}

\begin{corollary}
    Consider the iterates $\bx_0,\ldots,\bx_{K_{\star}}$ generated by Algorithm~\ref{alg:random-NC-PL-meta}. With probability at least $1-\eta$, 
    there exists some $k_{\star}=0,\ldots,K_{\star}-1$ such that
    $$
        \|\nabla (h_p)_{{1}/{(2\rho)}}(\bx_{k_{\star}+1})\| \le \varepsilon_1
        \quad\text{and}\quad
        \min_{\bw\in\bbS^n}\max_{\bA\in\partial_{\delta}^2 (h_p)_{{1}/{(2\rho)}}(\bx_{k_{\star}+1})}\langle\bA\bw,\bw\rangle\ge-\varepsilon_2.
    $$
\end{corollary}

\begin{proof}
    It suffices to combine Lemma~\ref{lma:bilevel-weakly-convex} with Corollary~\ref{cor:weakly-convex-Goldstein-SOSP}.
\end{proof}

Similar to Section~\ref{sec:weakly-convex-Goldstein-computation}, the only additional ingredient needed in the implementation is the evaluation of $\nabla (h_p)_{{1}/{(2\rho)}}$, or equivalently
$$
    \operatorname{Prox}_{{h_p}/{(2\rho)}}(\bx)=\arg\min\bigl\{h_p(\by)+\rho\|\by-\bx\|^2:\by\in\R^n\bigr\}.
$$
While this is a strongly convex program, typically, we do not have access to $\partial h_p$, as $h_p$ arises as the marginal function of a nonconvex program. This hinders the application of the subgradient method. Fortunately, the evaluation subroutine for $h_p$ in~\cite[Algorithm~2]{chen2025set} allows us to apply the zeroth-order algorithm in~\cite[Algorithm~1]{shamir2017optimal} (see also~\cite[Corollary~2]{shamir2017optimal}) to compute $\operatorname{Prox}_{{h_p}/{(2\rho)}}(\bx)$. We leave the details to interested readers.
As an aside, under further assumptions (e.g., the availability of a local
upper Lipschitz selection of $\operatorname{Arg}\max\bigl\{u(\bx,\by):\by\in\bbY^{\star}(\bx)\bigr\}$), one may also apply the subdifferential calculi for marginal functions developed in, e.g.,~\cite{mordukhovich2009subgradients,phuoc2025regular}, to compute $\partial h_p$, thereby enabling the use of first-order methods.


It is worth noting that in the literature, there is another NC--P\L{} regime, albeit with a slightly stronger assumption (i.e., a uniform P\L{} property of $\sigma u+\ell$ (in $\by$) for all sufficiently small $\sigma\ge 0$, rather than merely for $\ell$ itself as in Lemma~\ref{lma:bilevel-weakly-convex}, which corresponds to $\sigma=0$; see~\cite[Assumption~4.1(a)]{chen2024finding}),
thereby leading to an $L$-smooth hyperobjective function~\cite[Lemma~4.4]{chen2024finding} as opposed to Lemma~\ref{lma:bilevel-weakly-convex}.
However, this pretty much places us back in the framework of Section~\ref{sec:NC-SC-bilevel},
and as the hypergradient can be approximated tractably to an arbitrary accuracy (see~\cite[Lemmas~4.2--4.3]{chen2024finding} and also~\cite{shen2023penalty,kwon2024penalty}),
Algorithm~\ref{alg:random-NC-SC-meta} can be reused here almost verbatim. We leave the details to interested readers.

\subsubsection{Towards new regimes with weakly convex hyperobjective functions}
It is well known that the primal function of a WC--C (i.e., weakly-convex--concave) minimax program is weakly convex; see, e.g., the discussions preceding~\cite[Definition~1]{li2025nonsmooth}. However, it is also well known that minimax optimization is a special case of bilevel optimization; i.e., with $u=-\ell$. How does this minimax fact generalize to bilevel programs? Very naturally, one may reinterpret the condition as requiring both $u$ and $-\ell$ to be WC--C. 
However, perhaps surprisingly, even both $u$ and $-\ell$ being SC--SC 
is not enough, no matter whether pessimistic or optimistic.

\begin{example}
    Let $\bbX=\R$ and $\bbY=[-2,2]$. Consider the functions $u,\ell:\R^2\rightarrow\R$ given by
    $$
        u(x,y):=\frac{1}{2} x^2-\frac{1}{2} y^2+x y
        \quad\text{and}\quad
        \ell(x,y):=\frac{1}{2}(y-\sin x)^2-2 x^2,
    $$
    respectively. It can be calculated that 
    $$
        \nabla^2 u(x,y)=\begin{pmatrix}1&1\\1&-1\end{pmatrix}
        \quad\text{and}\quad
        \nabla^2 \ell(x,y)=\begin{pmatrix} \cos 2x + y\sin x-4 & -\!\cos x\\ -\!\cos x & 1 \end{pmatrix}.
    $$
    As a result, both $u$ and $-\ell$ are $1$-SC--$1$-SC on $\bbX\times\bbY$ as
    $$
        \cos 2x + y\sin x-4 \le 1+2\cdot 1 -4=-1
        \quad\text{for all $x\in\R$ and $y\in\bbY=[-2,2]$}.
    $$
    It is also easy to see that
    $$
        y^{\star}(x)=\sin x
        \quad\text{for all $x\in\R$},
    $$
    as $\sin x\in[-1,1]\subseteq[-2,2]=\bbY$. (In particular, this implies that $h_o=h_p$.) As a result, 
    $$
        h(x)=u\bigl(x,y^{\star}(x)\bigr)=u(x,\sin x)=\frac{1}{2} x^2-\frac{1}{2} \sin^2 x+x \sin x,
    $$
    and thus
    $$
        h^{\prime\prime}(x)=-x \sin x+2 \cos x-
        \cos 2x
        +1,
    $$
    which satisfies that
    $$
        \liminf_{x\rightarrow\infty}h^{\prime\prime}(x)\le
        \lim_{k\rightarrow\infty}h^{\prime\prime}(\pi/2+2\pi k)=\lim_{k\rightarrow\infty}\bigl(2-(\pi/2+2\pi k)\bigr)=-\infty;
    $$
    i.e., $h$ is not weakly convex.
\end{example}

While we do not yet have a satisfactory answer to the generalization, the following proposition, which exploits the constancy of $u(\bx,\bullet)+\nu\ell(\bx,\bullet)$ on $\bbY^{\star}(\bx)$ for some $\nu\ge 0$, may offer a useful clue. 

\begin{proposition}\label{prop:WC-C-bilevel}
    Suppose that $\bbX$ is convex, that $\bbY^{\star}(\bx)\neq\emptyset$ for every $\bx\in\bbX$, that there exist a constant $\nu\ge 0$ and a $\rho_p$-weakly convex function $p:\bbX\rightarrow\R$ such that
    $$
        u(\bx,\by)+\nu\ell(\bx,\by)=p(\bx)
        \quad\text{for all $\bx\in\bbX$ and $\by\in\bbY^{\star}(\bx)$},
    $$
    and that $\ell(\bullet,\by)$ is $\rho_{\ell}$-weakly concave on $\bbX$ for all $\by\in\bbY$. 
    Then, $h$ is $(\rho_p+\nu\rho_{\ell})$-weakly convex on $\bbX$.
\end{proposition}

We note that by taking $\ell=-u$, $\nu=1$, and $p=0$, Proposition~\ref{prop:WC-C-bilevel} recovers the minimax case. As an aside, Proposition~\ref{prop:WC-C-bilevel} does not impose any regularity condition on $\bbY^{\star}:\bbX\rightrightarrows\bbY$ such as its set smoothness~\cite[Definition~3]{chen2025set}, thereby potentially broadening the applicability of~\cite[Theorem~4]{chen2025set}.
Moreover, it can also be taken as a recognition principle for disguised minimax problems, as long as $u(\bx,\bullet)+\nu\ell(\bx,\bullet)$ is in fact constant on all of $\bbY$, rather than merely on $\bbY^{\star}(\bx)$.

\begin{proof}
    Let $v(\bx):=\min_{\by\in\bbY}\ell(\bx,\by)$. We have $u(\bx,\by)=p(\bx)-\nu v(\bx)$ for all $\bx\in\bbX$ and $\by\in\bbY^{\star}(\bx)$; in particular, this implies that $h_o=h_p$. 
    On the other hand,
    by~\cite[Proposition~4.1(ii)]{vial1983strong}, it follows that $v$, as the pointwise infimum of $\ell(\bullet,\by)$ for $\by\in\bbY$, is $\rho_{\ell}$-weakly concave on $\bbX$. This, together with the weak convexity of $p$ and~\cite[Proposition~4.1(i)]{vial1983strong}, further implies that $h$ is $(\rho_p+\nu\rho_{\ell})$-weakly convex on $\bbX$,
    as desired.
\end{proof}

As an illustration,
Proposition~\ref{prop:WC-C-bilevel} 
allows us to identify the minimax essence of
an important class of bilevel problems in machine learning---learning with Fenchel--Young losses; see, e.g.,~\cite{blondel2019learning,blondel2020learning,martins2022sparse,blondel2022learning,rakotomandimby2024learning,roulet2025loss}.

\begin{example}[Learning with Fenchel--Young losses]
Let $\bF(\bullet;\bw):\R^{d_x}\rightarrow\R^{d_y}$ be a model parametrized by $\bw\in\R^n$, $\Omega:\R^{d_y}\rightarrow\R\cup\{\infty\}$ be a regularization function such that 
$$
    \max\bigl\{\langle\bd{\theta}, \bd{\mu}\rangle-\Omega(\bd{\mu}):\bd{\mu} \in \operatorname{dom}(\Omega)\bigr\}
$$ 
is attainable for all $\bd{\theta}\in\R^{d_y}$,\footnote{We remark that this holds for all choices of $\Omega$ in~\cite[Table~1]{blondel2020learning}.}
and $\{(\bx_i,\by_i)\}_{i=1}^m\in\R^{d_x}\times\bbY_{\operatorname{data}}$ be a set of data where $\bbY_{\operatorname{data}}\subseteq\operatorname{dom}(\Omega)
$.
Learning with Fenchel--Young losses pertains to the following optimization problem
\begin{equation}\label{eq:learning-FY}
    \min\mleft\{\sum_{i=1}^m L_{\Omega}\bigl(\bF(\bx_i;\bw) ; \by_i\bigr):\bw\in\R^n\mright\},
    \quad\text{where $L_{\Omega}(\bd{\theta} ; \by):=\Omega^*(\bd{\theta})+\Omega(\by)-\langle\bd{\theta}, \by\rangle$},
\end{equation}
and $\Omega^*$ is the Fenchel conjugate of $\Omega$.
As we have (see, e.g., the discussions after~\cite[Definition~2]{blondel2020learning})
$$
    L_{\Omega}(\bd{\theta} ; \by)=\Omega(\by)-\langle\bd{\theta}, \by\rangle-\Omega(\widehat{\by})+\langle\bd{\theta}, \widehat{\by}\rangle,
$$
where
$$
    \widehat{\by} \in {\operatorname{Arg}\max}\bigl\{\langle\bd{\theta}, \bd{\mu}\rangle-\Omega(\bd{\mu}):\bd{\mu} \in \operatorname{dom}(\Omega)\bigr\}
$$
can be arbitrary, (\ref{eq:learning-FY}) can be equivalently rewritten as the following bilevel program
\begin{equation}\label{eq:bilevel-FY}
    \min_{\bw\in\R^n}\underset{(\hat{\by}_i)_{i=1}^m \in \bbY^{\star}(\bw)}{\min\!/\!\max}
    u\bigl(\bw,(\widehat{\by}_i)_{i=1}^m\bigr):=
    \sum_{i=1}^m \bigl(\Omega(\by_i)-\langle\bF(\bx_i;\bw), \by_i\rangle-\Omega(\widehat{\by}_i)+\langle\bF(\bx_i;\bw), \widehat{\by}_i\rangle\bigr),
\end{equation}
where
$$
    \bbY^{\star}(\bw):={\operatorname{Arg}\min}\mleft\{
    \ell\bigl(\bw,(\bd{\mu}_i)_{i=1}^m\bigr):=
    -\sum_{i=1}^m\bigl(\langle\bF(\bx_i;\bw), \bd{\mu}_i\rangle-\Omega(\bd{\mu}_i)\bigr):(\bd{\mu}_i)_{i=1}^m \in \operatorname{dom}(\Omega)^m\mright\}\neq\emptyset,
$$
where the nonemptiness follows from the attainability assumption. Now, we calculate that
\begin{equation}\label{eq:u+l-FY}
    u\bigl(\bw,(\widehat{\by}_i)_{i=1}^m\bigr)+\ell\bigl(\bw,(\widehat{\by}_i)_{i=1}^m\bigr)=\sum_{i=1}^m \bigl(\Omega(\by_i)-\langle\bF(\bx_i;\bw), \by_i\rangle\bigr)
    \quad\text{for all $\bw\in\R^n$ and $(\widehat{\by}_i)_{i=1}^m \in \operatorname{dom}(\Omega)^m$},
\end{equation}
which is independent of $\widehat{\by}_1,\ldots,\widehat{\by}_m$. As a result, as long as 
$$
    \bw\mapsto\sum_{i=1}^m \bigl(\Omega(\by_i)-\langle\bF(\bx_i;\bw), \by_i\rangle\bigr)
$$
is weakly convex and 
$$
    \bw\mapsto-\sum_{i=1}^m\bigl(\langle\bF(\bx_i;\bw), \bd{\mu}_i\rangle-\Omega(\bd{\mu}_i)\bigr)
$$
is weakly concave for all $(\bd{\mu}_i)_{i=1}^m \in \operatorname{dom}(\Omega)^m$ (with uniform moduli),
Proposition~\ref{prop:WC-C-bilevel} would imply that the hyperobjective function is weakly convex.
In particular, this is the case when
$$
    \bw\mapsto\sum_{i=1}^m\langle\bF(\bx_i;\bw), \bd{\mu}_i\rangle
$$
is $L$-smooth for all $(\bd{\mu}_i)_{i=1}^m \in \operatorname{dom}(\Omega)^m$, which can be achieved when, e.g., $\operatorname{dom}(\Omega)$ is bounded (e.g., perceptron, logistic, hinge, sparsemax, logistic (one-vs-all), structured perceptron, structured hinge, CRF, and sparseMAP losses; see~\cite[Table~1]{blondel2020learning}) and $\bw\mapsto\bF(\bx_i;\bw)$ for $i=1,\ldots,m$ are $L^{\prime}$-smooth (e.g., two-layer linear neural networks $\bF_{\operatorname{LNN}}(\bx;\bW_1,\bW_2):=\bW_2\bW_1\bx$~\cite{arora2019convergence}, and degree-$2$ factorization machines\footnote{We omit the linear term $\bw^{\T}\bx$ here as it can be incorporated by augmenting $\bx\leftarrow 1\vee\bx$; see, e.g.,~\cite[Section~8]{blondel2016polynomial}.} $f_{\operatorname{FM}}(\boldsymbol{x}; \boldsymbol{P}):=\sum_{j>i}\bigl(\boldsymbol{P}\boldsymbol{P}^{\T}\bigr)_{i j} x_{i} x_{j}$~\cite{blondel2016polynomial,blondel2016higher,blondel2017multi,atarashi2021factorization}).

In fact, (\ref{eq:bilevel-FY}) is a essentially a minimax problem in bilevel disguise: By (\ref{eq:u+l-FY}), we have
$$
    \underset{(\hat{\by}_i)_{i=1}^m \in \bbY^{\star}(\bw)}{\min\!/\!\max}
    u\bigl(\bw,(\widehat{\by}_i)_{i=1}^m\bigr)=
    \max_{(\hat{\by}_i)_{i=1}^m \in \operatorname{dom}(\Omega)^m}
    \sum_{i=1}^m \bigl(\Omega(\by_i)-\langle\bF(\bx_i;\bw), \by_i\rangle-\Omega(\widehat{\by}_i)+\langle\bF(\bx_i;\bw), \widehat{\by}_i\rangle\bigr).
$$
\end{example}

Another illustration is data-driven inverse optimization with suboptimality losses; see, e.g.,~\cite{mohajerin2018data}.

\begin{example}[Inverse optimization with suboptimality losses]
    Let $f:\R^{d_1}\times\R^{d_2}\rightarrow\R$ be a ground-truth function to be recovered, $\bbU\subseteq\R^{d_1}$, $\bbV:\R^{d_1}\rightrightarrows\R^{d_2}$, $\doublestroke{\Xi}:=\{(\bu,\bv):\bu\in\bbU,\,\bv\in\bbV(\bu)\}$,
    $$
        \bbV^{\star}(\bu):=\operatorname{Arg}\min\bigl\{f(\bu,\bv):\bv\in\bbV(\bu)\bigr\},
    $$
    which is assumed to be nonempty for all $\bu\in\bbU$ following~\cite[Section~2]{mohajerin2018data}, 
    $$
        \bbF:=\left\{\widetilde{f}(\bullet,\bullet;\bd{\theta}):\R^{d_1}\times\R^{d_2}\rightarrow\R:\bd{\theta}\in\doublestroke{\Theta}\right\},
        \quad\text{where $\doublestroke{\Theta}\subseteq\R^n$}
    $$
    be the hypothesis space where each model $\widetilde{f}(\bullet,\bullet;\bd{\theta})$ is parametrized by $\bd{\theta}\in\doublestroke{\Theta}$, and $(\widehat{\bu}_i,\widehat{\bv}_i)\in\bbU\times\bbV^{\star}(\widehat{\bu}_i)$ for $i=1,\ldots,m$ be $m$ independent samples from a probability distribution $\bbP$ supported on $\doublestroke{\Xi}$.
    Inverse optimization with suboptimality losses~\cite[Definition~2.3]{mohajerin2018data} pertains to the following bilevel program
    \begin{equation}\label{eq:bilevel-IO}
        \min_{\bd{\theta}\in\doublestroke{\Theta}}\underset{(\bv_i)_{i=1}^m \in \bbY^{\star}(\bd{\theta})}{\min\!/\!\max}
        u\bigl(\bd{\theta},(\bv_i)_{i=1}^m\bigr):=\frac{1}{m} \sum_{i=1}^m \left(\widetilde{f}(\widehat{\bu}_i, \widehat{\bv}_i; \bd{\theta})- \widetilde{f}(\widehat{\bu}_i, \bv_i; \bd{\theta})\right),
    \end{equation}
    where
    $$
        \bbY^{\star}(\bd{\theta}):=\operatorname{Arg}\min\mleft\{
        \ell\bigl(\bd{\theta},(\bv_i)_{i=1}^m\bigr):=
        \frac{1}{m} \sum_{i=1}^m  \widetilde{f}(\widehat{\bu}_i, \bv_i; \bd{\theta}):(\bv_i)_{i=1}^m \in \prod_{i=1}^m\bbV(\widehat{\bu}_i)\mright\},
    $$
    which is also assumed to be nonempty.\footnote{We remark that this is the case when, e.g., $\bbV$ is compact-valued and $\widetilde{f}(\bullet,\bullet;\bd{\theta})$ is continuous.}
    Now, we calculate that
    \begin{equation}\label{eq:u+l-IO}
        u\bigl(\bd{\theta},(\bv_i)_{i=1}^m\bigr)+\ell\bigl(\bd{\theta},(\bv_i)_{i=1}^m\bigr)=\frac{1}{m} \sum_{i=1}^m \widetilde{f}(\widehat{\bu}_i, \widehat{\bv}_i; \bd{\theta})
        \quad\text{for all $\bd{\theta}\in\doublestroke{\Theta}$ and $(\bv_i)_{i=1}^m \in \prod_{i=1}^m\bbV(\widehat{\bu}_i)$},
    \end{equation}
    which is independent of $\bv_1,\ldots,\bv_m$. As a result, as long as $\doublestroke{\Theta}$ is convex,
    $$
        \bd{\theta}\mapsto \frac{1}{m} \sum_{i=1}^m \widetilde{f}(\widehat{\bu}_i, \widehat{\bv}_i; \bd{\theta})
    $$
    is weakly convex on $\doublestroke{\Theta}$, and 
    $$
        \bd{\theta}\mapsto \frac{1}{m} \sum_{i=1}^m  \widetilde{f}(\widehat{\bu}_i, \bv_i; \bd{\theta})
    $$
    is weakly concave on $\doublestroke{\Theta}$ for all $(\bv_i)_{i=1}^m \in \prod_{i=1}^m\bbV(\widehat{\bu}_i)$ (with uniform moduli),
    Proposition~\ref{prop:WC-C-bilevel} would imply that the hyperobjective function is weakly convex on $\doublestroke{\Theta}$.
    In particular, the latter two properties hold when, e.g., $\widetilde{f}(\bu, \bv; \bullet)$ is $L$-smooth for all $(\bu,\bv)\in\doublestroke{\Xi}$.
    
    In fact, (\ref{eq:bilevel-IO}) is a essentially a minimax problem in bilevel disguise: By (\ref{eq:u+l-IO}), we have
    $$
        \underset{(\bv_i)_{i=1}^m \in \bbY^{\star}(\bd{\theta})}{\min\!/\!\max}
        u\bigl(\bd{\theta},(\bv_i)_{i=1}^m\bigr)=
        \max_{(\bv_i)_{i=1}^m \in \prod_{i=1}^m\bbV(\hat{\bu}_i)}
        \frac{1}{m} \sum_{i=1}^m \left(\widetilde{f}(\widehat{\bu}_i, \widehat{\bv}_i; \bd{\theta})- \widetilde{f}(\widehat{\bu}_i, \bv_i; \bd{\theta})\right).
    $$
    
\end{example}


\subsubsection{Trilevel optimization and beyond}
In fact, Algorithm~\ref{alg:random-meta} is also applicable to the following trilevel problem (see, e.g.,~\cite{vicente1994bilevel,sato2021gradient,giovannelli2025stochastic,kent2025stochastic})
$$
    \min_{\bx_1\in\R^{n_1}}
    h(\bx_1):=
    f_1\mleft(\bx_1,\bx_2^{\star}(\bx_1),\bx_3^{\star}\bigl(\bx_1,\bx_2^{\star}(\bx_1)\bigr)\mright),
$$
where 
$$
    \bx_2^{\star}(\bx_1):=\underset{\bx_2\in\R^{n_2}}{\arg\min}\,
    f_2\bigl(\bx_1,\bx_2,\bx_3^{\star}(\bx_1,\bx_2)\bigr)
    \quad\text{and}\quad
    \bx_3^{\star}(\bx_1,\bx_2):=\underset{\bx_3\in\R^{n_3}}{\arg\min}\,
    f_3(\bx_1,\bx_2,\bx_3),
$$
$f_1,f_2,f_3:\R^{n_1}\times\R^{n_2}\times\R^{n_3}\rightarrow\R$ are the upper, middle, and lower functions,
and the notation tacitly assumes that both the minima are attainable and unique, as will be the case in the sequel.
Indeed, under~\cite[Assumptions~3.1--3.3]{giovannelli2025stochastic} (which we prefer not spell out here as they are technically involved and entail new notation),
\cite[Proposition~E.2]{giovannelli2025stochastic} (or more precisely,~\cite[(E.11)]{giovannelli2025stochastic}) implies that $h$ is $L$-smooth for some $L>0$, 
albeit, to the best of our knowledge,
its order is not yet fully clear.
Besides, $\nabla h$ also admits an explicit closed-form expression, evaluable via matrix multiplications; see~\cite[Proposition~A.1]{giovannelli2025stochastic}. As a result, Algorithm~\ref{alg:random-meta} can be applied to $h$ exactly as in Section~\ref{sec:NC-SC-bilevel}; we leave the details to interested readers.

What is the order of $L$? 
When is $h$ weakly convex?
How do the properties carry over to multilevel settings?
We leave these intriguing questions to future work.


\section{Concluding remarks}\label{sec:conclusions}
In this paper, we have exhibited a randomized algorithm (i.e., Algorithm~\ref{alg:random-meta}) that computes Goldstein approximate SOSPs of
$L$-smooth functions in 
a polynomial number of
first-order oracle calls. In particular, in the regime where $\varepsilon_1,\varepsilon_2,\delta\asymp\varepsilon$, the complexity admits the explicit form 
$
    \widetilde{O}\bigl({\Delta L^8 n^2}/{\varepsilon^9}+{\Delta L^6 n^3}/{\varepsilon^7}\bigr)
$.
In the course of developing Algorithm~\ref{alg:random-meta}, some underpinning theory is also established; e.g., a framework for second-order randomized smoothing, the corresponding gradient and Hessian estimators, and a few properties of the Goldstein generalized Hessian, which 
may 
be of independent interest.
Furthermore,
with Algorithm~\ref{alg:random-meta} as a foundation, we extend the computability of Goldstein SOSPs to weakly convex functions, NC--SC and NC--P\L{} bilevel programs, and beyond,
together with some auxiliary results that may be of independent interest, such as a second-order stationarity coincidence between a weakly convex function and its Moreau envelope.

This work also points to some directions for future work,
as detailed in the following subsections.

\subsection{Dimension-free computations of Goldstein approximate SOSPs}
It is well known that for Lipschitz functions, Goldstein approximate FOSPs are computable by randomized algorithms with a dimension-free complexity; see, e.g.,~\cite{zhang2020complexity,tian2022finite,davis2022gradient}. However, while 
Algorithm~\ref{alg:random-meta} is able to compute Goldstein approximate SOSPs of $L$-smooth functions, its complexity depends 
on $n$.
This naturally raises the following question: 
$$
    \boxed{\textnormal{Whether Goldstein approximate SOSPs remain computable in a dimension-free manner?}}
$$
In fact, prior to developing the algorithms in this paper, we had already devoted considerable efforts to the question, but to no avail.
The main obstacle that we encountered is the absence of an error bound for the first-order expansion of $\lambda_{\min}$ (see, e.g.,~\cite[Theorem~3.12]{hiriart1995sensitivity}) that remains valid even beyond the 
spectral-gap regimes,
whereas the $\ell_2$-norm by contrast admits an exact global quadratic expansion given by the binomial formula, which plays a crucial role in the proof of~\cite[Theorem~8]{zhang2020complexity}.
(In passing, we note that the best such result that we are aware of is~\cite[Proposition~1]{nakatsukasa2018off}; 
see also, e.g.,~\cite[Chapter~II, Example~3.3]{kato1995perturbation} for related error bounds.)
In fact, we are more inclined to believe that the answer is in the negative.
If this is indeed the case, then it would exhibit a 
phase transition
between the first-order and second-order regimes,
and~\cite[Algorithm~1]{zhang2020complexity} would amount to a miracle.

\begin{conjecture}
    No dimension-free randomized algorithm computes Goldstein approximate SOSPs of $L$-smooth functions.
\end{conjecture}

\subsection{Computing near-approximate SOSPs for more structured function classes}
While it is well known that computing a near-approximate FOSP of Lipschitz functions is intractable even with the help of randomization~\cite[Theorem~1]{kornowski2021oracle}, under the additional assumption of weak convexity, such a point can be computed efficiently by, e.g., the (deterministic) subgradient method~\cite[Theorem~3.1]{davis2019stochastic}.
It is thus natural to ask:
$$
    \boxed{
    \begin{gathered}
    \textnormal{Whether there exists another kind of regularity that, as weak convexity,}\\
    \textnormal{can enable the efficient computation of near-approximate SOSPs of $L$-smooth functions?}
    \end{gathered}
    }
$$
To be precise, by near-approximate SOSPs, we simply mean
$$
    \|\nabla f(\bx)\|\le\varepsilon_1
    \quad\text{and}\quad
    \max_{\by\in\bx+\delta\bbB^n}\min_{\bw\in\bbS^n}\max_{\bA\in\partial_{\dC}^2 f(\by)}\langle\bA\bw,\bw\rangle\ge-\varepsilon_2.
$$
For this question, we have also invested considerable efforts, albeit, again, with a pessimistic outlook:
Some evidence suggests that the resulting function classes should be almost indistinguishable from those of convex functions
and functions with Lipschitz Hessians, while we do not yet know how to formalize this intuition.

Of course, it is equally important to establish hardness results for computing such points.

\subsection{Landscape analysis based on the notion of Goldstein SOSPs}
In this paper, we have shown certain computability of Goldstein approximate SOSPs.
However, the second-order regime is somehow special, as at least in the smooth world many problems 
are devoid of
spurious SOSPs, 
as discussed in Section~\ref{sec:introduction}.
It is thus natural to wonder:
$$
    \boxed{
    \begin{gathered}
        \textnormal{Are there nonsmooth nonconvex functions, particularly those arising in applications,}\\
        \textnormal{that are devoid of spurious Goldstein SOSPs?}
    \end{gathered}
    }
$$
A potential starting point could be the objective function of $\ell_1$-norm rank-one symmetric matrix factorization (i.e., $\bu\mapsto\|\bu\bu^{\T}-\bu^{\star}{\bu^{\star}}^{\T}\|_1/2$), which we have shown in~\cite[Theorem~3.8]{guan2024ell_1} to be devoid of Rockafellar SOSPs. We believe that either a positive or a negative answer would be interesting in its own right.


\subsection{Further directions}
Of course, beyond the directions sketched earlier, there are certainly many others.
Some immediate topics that we have in mind are:
\begin{itemize}
    \item The role of the Goldstein generalized Hessian in convergence rates of algorithms; see, e.g.,~\cite{davis2025local,kong2025lipschitz,gebken2025analyzing} for its first-order counterparts, such as the Goldstein modulus~\cite[Definition~2.2]{kong2025lipschitz}.

    \item 
    What if constraints arise?
    For example, can we define second-order analogues of Goldstein FJ/KKT stationarity~\cite[Definitions~1--2]{grimmer2025goldstein}
    and develop corresponding algorithms? 

    \item Manifold extensions; see, e.g.,~\cite[Definition~3.1]{grohs2016varepsilon} for the Riemannian Goldstein subdifferential and~\cite{grohs2016varepsilon,sahinoglu2025finite} for the corresponding algorithms.
    In particular, this can potentially be applied to analyze the landscapes of nonsmooth problems posed on manifolds; e.g., dictionary learning~\cite{bai2019subgradient}.


    \item Hessian sampling; see, e.g.,~\cite{burke2005robust,davis2022gradient} for gradient-sampling algorithms.

    \item Other aspects, such as multiobjective optimization~\cite{gebken2021efficient,sonntag2024descent}, robust control~\cite{guo2022global}, etc.
\end{itemize}

\section*{Acknowledgments}
We are deeply indebted to Zhenwei Lin and Zhe (Jimmy) Zhang at Edwardson School of Industrial Engineering, Purdue University for an insightful discussion regarding the proof of~\cite[Proposition~2.3]{lin2022gradient} at an ancient stage of the project, which ultimately became, in essence, the proof of Theorem~\ref{thm:smoothing}.

\bibliographystyle{abbrv}
\bibliography{references.bib}{}




\end{document}